\documentclass[10pt, a4paper, leqno]{amsart}
\usepackage{amsfonts, amssymb, amsmath}
\usepackage[vmargin=3cm]{geometry}

\usepackage[T1]{fontenc}

\usepackage{lmodern}
\usepackage{microtype}

\DeclareMathAlphabet{\mathscr}{OT1}{pzc}{m}{it}

\usepackage{amsthm}
\usepackage{microtype}
\usepackage{enumerate}
\usepackage[pdftitle={Semi-classical states for the Choquard equation},pdfauthor={Vitaly Moroz and Jean Van Schaftingen}]{hyperref}
\usepackage[abbrev, backrefs]{amsrefs}

\newtheorem{theorem}{Theorem}

\newtheorem{proposition}{Proposition}[section]

\newtheorem{lemma}[proposition]{Lemma}

\theoremstyle{definition}

\newtheorem{case}{Case}
\numberwithin{equation}{section}
\DeclareMathOperator{\supp}{supp}

\newcommand{\N}{{\mathbb N}}
\newcommand{\R}{{\mathbb R}}

\newcommand{\I}{{\mathcal{I}}}
\newcommand{\J}{{\mathcal{J}}}

\newcommand{\norm}[1]{\lVert #1 \rVert}
\newcommand{\abs}[1]{\lvert #1 \rvert}
\newcommand{\bigabs}[1]{\bigl\lvert #1 \bigr\rvert}

\newcommand{\weakto}{\rightharpoonup}
\newcommand{\dualprod}[2]{\langle #1, #2 \rangle}
\newcommand{\scalprod}[2]{( #1 \vert #2)}
\DeclareMathOperator{\dist}{dist}
\newcommand{\dif}{\,\mathrm{d}}

\title{Semi-classical states for the Choquard equation}
\author{Vitaly Moroz}
\address{Swansea University\\ Department of Mathematics\\ Singleton Park\\
Swansea\\ SA2~8PP\\ Wales, United Kingdom}	
\email{V.Moroz@swansea.ac.uk}

\author{Jean Van Schaftingen}
\address{Universit\'e Catholique de Louvain\\ Institut de Recherche en Math\'ematique et Phy\-sique\\ Chemin du Cyclotron 2 bte L7.01.01\\ 1348 Louvain-la-Neuve \\ Belgium}
\email{Jean.VanSchaftingen@uclouvain.be}



\date{\today}

\begin{document}
\begin{abstract}
We study the nonlocal equation
\begin{equation*}
 - \varepsilon^2 \Delta u_\varepsilon + V u_\varepsilon  = \varepsilon^{-\alpha} \bigl(I_\alpha \ast \abs{u_\varepsilon}^p\bigr) \abs{u_\varepsilon}^{p - 2} u_\varepsilon\quad\text{in \(\R^N\)},
\end{equation*}
where \(N \ge 1\), \(\alpha \in (0, N)\), \(I_\alpha (x) = A_\alpha/\abs{x}^{N - \alpha}\)
is the Riesz potential and \(\varepsilon > 0\) is a small parameter.
We show that if the external potential \(V \in C (\R^N; [0, \infty))\) has a local minimum and \(p \in [2, (N + \alpha)/(N - 2)_+)\) then for all small \(\varepsilon > 0\) the problem has a family of solutions concentrating to the local minimum of \(V\) provided that: either \(p > 1 + \max (\alpha, \frac{\alpha + 2}{2})/(N - 2)_+\), or \(p > 2\) and \(\liminf_{\abs{x} \to \infty} V (x) \abs{x}^2 > 0\), or \(p = 2\) and \(\inf_{x \in \R^N} V (x) (1 + \abs{x}^{N - \alpha})>0\).
Our assumptions on the decay of \(V\) and admissible range of \(p\ge 2\) are optimal.
The proof uses variational methods and a novel nonlocal penalization technique that we develop in this work.
\end{abstract}

\maketitle

\tableofcontents

\section{Introduction and results}

\subsection{Introduction}
We consider the nonlinear Choquard equation
\begin{equation}
\tag{$\mathcal{P}_\varepsilon$}
\label{equationNLChoquard}
 - \varepsilon^2 \Delta u_\varepsilon + V u_\varepsilon  = \varepsilon^{-\alpha} \bigl(I_\alpha \ast \abs{u_\varepsilon}^p\bigr) \abs{u_\varepsilon}^{p - 2} u_\varepsilon\quad\text{in \(\R^N\)},
\end{equation}
where the dimension \(N \in \N=\{1,2,\dotsc\}\) is given,
\(V\in C(\R^N,[0,\infty))\) is an external potential,
\(I_\alpha : \R^N\setminus\{0\} \to \R\) is the Riesz potential \cite{Riesz} of order \(\alpha \in (0, N)\) defined for every \(x \in \R^N \setminus \{0\}\) by
\[
  I_\alpha(x)=\frac{\Gamma(\tfrac{N-\alpha}{2})}
                   {\Gamma(\tfrac{\alpha}{2})\pi^{N/2}2^{\alpha} \abs{x}^{N-\alpha}},
\]
\(p \in \big[2,\frac{N+\alpha}{(N-2)_+}\big)\) and \(\varepsilon > 0\) is a small parameter.

For \(N = 3\), \(\alpha = 2\) and \(p = 2\), the equation \eqref{equationNLChoquard} is the  \emph{Choquard--Pekar equation}
which goes back to the description of the quantum theory of a polaron at rest by S.\thinspace I.\thinspace Pekar in 1954 \cite{Pekar1954} and which reemerged in 1976 in the work of P.\thinspace Choquard on the modeling of an electron trapped in its own hole used \eqref{equationNLChoquard},
in a certain approximation to Hartree--Fock theory of one-component plasma \cite{Lieb1977}.
The equation \eqref{equationNLChoquard} was also proposed in the 90's by K.\thinspace R.\thinspace W.\thinspace Jones \cite{KRWJones1995gravitational,KRWJones1995newtonian} and R.\thinspace Penrose \cite{Penrose1996}
as a model of self-gravitating matter and is known in that context as the \emph{Schr\"odin\-ger--Newton equation} \cite{Moroz-Penrose-Tod-1998}.
Finally, if \(u\) is a solution of \eqref{equationNLChoquard} with \(V\equiv 1\) then the wave-function \(\psi\) defined by \(\psi(t,x)=e^{-it/\varepsilon} u(x)\)
is a solitary wave of the focusing time-dependent Hartree equation
\[
 i \varepsilon\psi_t + \varepsilon^2\Delta \psi = -\varepsilon^{-\alpha}(I_\alpha \ast \abs{\psi}^2)\psi\quad\text{in \(\R \times\R^N\)};
\]
the equation \eqref{equationNLChoquard} is thus the \emph{stationary nonlinear Hartree equation}.

When the parameter \(\varepsilon > 0\) is fixed, the existence and qualitative properties of solutions of problem \eqref{equationNLChoquard} have been studied for a few decades by variational methods \citelist{\cite{Lieb1977}\cite{Lions1980}\cite{Lions1984-1}*{Chapter III}\cite{Menzala1980}\cite{Menzala1983}}. Recently, for fixed \(\varepsilon > 0\) and when \(V\) converges exponentially to a positive limit at infinity, the existence at least one positive solution to \eqref{equationNLChoquard} has been proved \citelist{\cite{ClappSalazar}\cite{CingolaniClappSecchi2012}*{theorem~1.2}}.

In quantum physical models, the parameter \(\varepsilon\) is usually the adimensionalized Planck constant, which is generically quite small. In quantum physics it is expected that in the semi-classical limit \(\varepsilon \to 0\) the dynamics should be governed by the classical external potential \(V\). In particular, there should be a correspondence between
\emph{semi-classical} solutions of the equation \eqref{equationNLChoquard} and critical points of the potential \(V\).

Mathematically, it can be observed that if \(u_\varepsilon\) is a solution of \eqref{equationNLChoquard} and \(a_\varepsilon \in \R^N\), then the rescaled function \(v_\varepsilon : \R^N \to \R\) defined for \(y \in \R^N\) by
\[
  v_\varepsilon (y) = u_\varepsilon (a_\varepsilon +\varepsilon y)
\]
satisfies the \emph{rescaled equation}
\[
  - \Delta v_\varepsilon + V(a_\ast+\varepsilon\,\cdot) v_\varepsilon = \bigl(I_\alpha \ast \abs{v_\varepsilon}^p\bigr) \abs{v_\varepsilon}^{p - 2} v_\varepsilon\quad\text{in \(\R^N\)}.
\]
This suggests some convergence, as \(\varepsilon\to 0\), of the family of rescaled solutions \((v_\varepsilon)_{\varepsilon > 0}\) to a  solution \(v_*\) of the \emph{limiting rescaled equation}
\begin{equation}
\label{eqLimit}\tag{$\mathcal{P}_*$}
  - \Delta v_* + V(a_\ast) v_* = \bigl(I_\alpha \ast \abs{v_*}^p\bigr) \abs{v_*}^{p - 2} v_*\quad\text{in \(\R^N\)}.
\end{equation}

For the local nonlinear Schr\"odinger equation
\begin{equation}\label{equationLocal}
 - \varepsilon^2 \Delta u_\varepsilon + V u_\varepsilon = \abs{u_\varepsilon}^{p - 2} u_\varepsilon\quad\text{in \(\R^N\)},
\end{equation}
which is the formal limit of the Choquard equation \eqref{equationNLChoquard} as \(\alpha \to 0\), such families of  semi-classical solutions that concentrate, as \(\varepsilon\to 0\), around one or several critical points of the potential \(V\) have been constructed by various methods during the last decades \citelist{\cite{FloerWeinstein}\cite{delPinoFelmer1997}\cite{delPinoFelmer1998}\cite{AmbrosettiBadialeCingolani}\cite{AmbrosettiMalchiodi2007}\cite{AmbrosettiMalchiodi2006}}.

The question of the existence of semi-classical solutions for the nonlocal problem \eqref{equationNLChoquard} has been posed more recently \cite{AmbrosettiMalchiodi2007}*{p.~29}. In the case  \(N = 2\), \(\alpha = 2\), \(p = 2\),
families of solutions have been constructed by a \emph{Lyapunov--Schmidt type reduction} when \(\inf V > 0\) by Wei Juncheng and M.\thinspace Winter \cite{WeiWinter2009} (see also \citelist{\cite{CingolaniSecchi}}) and when \(V > 0\) and \(\liminf_{\abs{x} \to \infty} V (x) \abs{x}^\gamma > 0\) for some \(\gamma \in [0, 1)\) by S.\thinspace Secchi \cite{Secchi2010}\footnote{It should be noted that lemma 15 in \cite{Secchi2010} only holds when there exists \(\gamma < 1\) such that \(\liminf_{\abs{x} \to \infty} V (x) \abs{x}^\gamma > 0\) (in the first term in (10) therein, one should read \(v(y)\) instead of \(v (x)\)). It is known that if \(\limsup_{\abs{x} \to \infty} V (x) \abs{x}^\gamma = 0\) for some \(\gamma > 1\), the problem (1) considered in \cite{Secchi2010} does not have a positive solution \cite{MVSNENLSNE}*{theorem 3}.}.
This method of construction depends on the existence, uniqueness and nondegeneracy up to translations of the positive solution of the limiting equation \eqref{eqLimit},
which is a difficult problem that has only been fully solved in the case \(N = 3\), \(\alpha = 2\) and \(p = 2\) \citelist{\cite{Lieb1977}\cite{WeiWinter2009}}.

S. Cingolani, S. Secchi and M. Squassina have proved the existence of solutions concentrating around several minimum points of \(V\) by a global penalization method for \(N = 3\), \(p = 2\) and \(\alpha = 2\) \cite{CingolaniSecchiSquassina2010};
in their work the Laplacian can be perturbed by a magnetic field and the Riesz potential
can be replaced by any potential homogeneous of degree \(-1\).

In the present work we prove the existence of a family of positive solutions to \eqref{equationNLChoquard}
which concentrate as \(\varepsilon\to 0\) to a local minimum of the potential \(V\),
in arbitrary dimensions \(N \in \N\), with Riesz potentials \(I_\alpha\) of any order \(\alpha \in (0, N)\),
and general external potentials \(V \in C(\R^N;[0, \infty))\) without any a-priori restrictions
on the decay or growth of \(V\) at infinity and for an optimal range of exponents \(p\ge 2\).
Our proofs use variational methods and a novel \emph{nonlocal penalization technique} that we develop in this work.
For convenience, we formulate our main results separately for the cases \(p=2\) and \(p>2\)

\subsection{The locally linear case $p=2$}
For the classical Choquard nonlinearity \(p = 2\) our main existence and concentration result is the following.

\begin{theorem}
\label{theoremLinear}
Let \(N \in \N\), \(\alpha \in ((N-4)_+, N)\), \(p = 2\) and \(V \in C(\R^N;[0, \infty))\).
Assume that either \(\alpha < N - 2\) or
\[
 \inf_{x \in \R^N} V (x) \bigl(1 + \abs{x}^{N - \alpha}\bigr) > 0.
\]
If \(\Lambda \subset \R^N\) is an open bounded set such that
\[
  0 < \inf_{\Lambda} V < \inf_{\partial \Lambda} V,
\]
then problem \eqref{equationNLChoquard} has a family of positive solutions \((u_\varepsilon)_{\varepsilon \in (0, \varepsilon_0)}\in H^1_{V}(\R^N)\) such that for
a family of points \((a_\varepsilon)_{\varepsilon \in (0, \varepsilon_0)}\) in \(\Lambda\) and for every \(\rho > 0\),
\begin{align*}
  \lim_{\varepsilon \to 0} V (a_\varepsilon) & = \inf_{\Lambda} V,&
  \liminf_{\varepsilon \to 0} \varepsilon^{-N} \int_{B_{\varepsilon \rho} (a_\varepsilon)} \abs{u_\varepsilon}^2 &> 0,&
  \lim_{\substack{R \to \infty\\ \varepsilon \to 0}} \norm{u_\varepsilon}_{L^{\infty} (\R^N \setminus B_{\varepsilon R} (a_\varepsilon))}  &= 0.
\end{align*}
\end{theorem}

Unless otherwise stated we understand solutions of Choquard equation \eqref{equationNLChoquard} in the weak sense,
as critical points of the energy functional
\[
\mathcal{E}_\varepsilon(u)= \frac{1}{2}\int_{\R^N} \Big( \varepsilon^2\abs{\nabla u}^2 + V\abs{u}^2\Big)
 - \frac{1}{2p \varepsilon^\alpha} \int_{\R^N} \bigabs{I_\frac{\alpha}{2}\ast \abs{u}^p}^2,
\]
which is formally associated to \eqref{equationNLChoquard}. Because we do not impose any a-priori assumptions
on the decay of the external potential \(V\) at infinity, \(\mathcal{E}_\varepsilon\) may not be well-defined on the natural Sobolev space \(H^1_{V}(\R^N)\) associated with the local quadratic part of the energy (see section \ref{sect-H-V-definition} below for a precise definition).
The nonlocal penalization method developed in this work allows, in particular, to modify the nonlinearity in such a way that the penalized variational problem becomes well-posed in the space \(H^1_{V}(\R^N)\). We shall comment on some features of the nonlocal penalization method in section~\ref{sect-penalization}.

When \(N \le 2\), the condition on the potential \(V\) is always
\[
  \inf_{x \in \R^N} V (x) \bigl(1 + \abs{x}^{N - \alpha}\bigr) > 0.
\]
In the original Choquard's equation case \(N = 3\) and \(\alpha = 2\) we include the critical decay case \(\inf_{x \in \R^3} V (x) (1 + \abs{x}) > 0\), which is not covered by the arguments in \cite{Secchi2010}.
On a more technical level we also remove the boundedness assumptions on \(V\) and its derivative and replace the assumption of the existence of a nondegenerate critical point of \(V\) by the existence of a strict local minimum of \(V\).
Our proof of theorem \ref{theoremLinear} uses variational methods and
a novel nonlocal penalization technique that we develop in this work, which can be thought as a counterpart of the penalization scheme for the nonlinear Schr\"odinger equation \eqref{equationLocal} \citelist{\cite{BonheureVanSchaftingen2008}\cite{delPinoFelmer1997}\cite{MorozVanSchaftingen2010}}.

A bootstrap argument similar to \cite{MVSGNLSNE}*{Proposition 4.1} shows that
the solutions \((u_\varepsilon)_{\varepsilon \in (0, \varepsilon_0)}\) constructed in theorem \ref{theoremLinear} are classical, in the sense that \(u_\varepsilon\in C^2(\R^N)\).
We emphasize however that our penalization scheme works entirely at the level of weak solutions and does not require any additional a-priori regularity.

The assumptions of theorem \ref{theoremLinear} are optimal in the following sense.
First, the condition \(\alpha \le N - 4\) is justified by the fact that otherwise the limiting equation \eqref{eqLimit} has no finite energy solutions.
\begin{theorem}
\label{theoremExistenceLimitLinear}
Let \(N\in\N\), \(\alpha \in (0, N)\) and \(p = 2\). Then the limiting equation \eqref{eqLimit} has a nontrivial solution in the space \(H^1(\R^N) \cap C^2 (\R^N)\) if and only if
\(\alpha > N - 4\).
\end{theorem}

The sufficiency for the existence of solutions in the above theorem traces back to the work of P.-L. Lions \cite{Lions1980} (see also \citelist{\cite{MVSGNLSNE}\cite{GenevVenkov2012}\cite{Ma-Zhao-2010}}), the necessity part follows from the Poho\v zaev identity for  \eqref{eqLimit} \citelist{\cite{Menzala1983}\cite{CingolaniSecchiSquassina2010}*{lemma 2.1}\cite{MVSGNLSNE}*{theorem 2}\cite{MVSGGCE}*{proposition~3.5}}.
For \(\alpha \le N - 4\) theorem \ref{theoremExistenceLimitLinear} does not exclude solutions concentrating
around a zero of a nonnegative potential. Such solutions were constructed for the local nonlinear Schr\"odinger equation \citelist{\cite{ByeonWang2002}\cite{ByeonWang2003}}.

The restriction on the decay of \(V\) at infinity when \(\alpha > N - 2\) is optimal
in view of a nonlinear Liouville theorem for supersolutions of \eqref{equationNLChoquard}.

\begin{theorem}[Moroz  and Van Schaftingen \cite{MVSNENLSNE}*{theorems 1 and 3}]
\label{theoremLiouvilleLinear}
Let \(N \in \N\), \(\alpha \in (0, N)\), \(p = 2\), \(V \in C(\R^N;[0, \infty))\) and \(\varepsilon > 0\).
If \eqref{equationNLChoquard} admits a positive distributional supersolution in \(\R^N\) and \(\alpha > N - 2\), then for every \(\gamma > N - \alpha\),
\[
  \limsup_{\abs{x} \to \infty} V (x) \abs{x}^{\gamma} > 0.
\]
\end{theorem}

For the local nonlinear Schr\"odinger equation the previous two obstructions --- nonexistence of finite-energy solutions to the limiting problem and Liouville theorems in outer domains --- are the only one.
For the Choquard equation in the r\'egime \(\alpha \ge N - 2\) we observe a new and essentially nonlocal phenomenon of a \emph{strong critical potential well}: if the potential \(V\) vanishes somewhere in \(\R^N\setminus\Lambda\) significantly enough,
then problem \eqref{equationNLChoquard} does not have a family of solutions \emph{that concentrates inside \(\Lambda\)}.

\begin{theorem}\label{theoremNonConcentrationSpecific}
Let \(N \in \N\), \(\alpha \in (0, N)\), \(p = 2\) and \(V \in C(\R^N;[0, \infty))\).
Assume that \(\alpha > N - 2\) and, as \(x \to a_*\),
\begin{equation}\label{vanishingrate}
 V (x)= o \big(\abs{x - a_*}^{\frac{4}{\alpha + 2 - N}-2}\big).
\end{equation}
If \((u_\varepsilon)_{\varepsilon \in (0, \varepsilon_0)}\) is a family in \(H^1_\mathrm{loc}(\R^N)\cap L^2((1+\abs{x})^{-(N-\alpha)}\dif x)\) of positive solutions of \eqref{equationNLChoquard} then for every compact set \(K\subset \R^N\setminus\{a_*\}\), as \(\varepsilon\to 0\),
\[
 \int_{K} \abs{u_\varepsilon}^2 = o (\varepsilon^N).
\]
\end{theorem}

In this case, the nonlocal interaction with the critical potential well
forces the rescaled mass to vanish outside \(a_*\).

The borderline case \(\alpha = N - 2\) is the most delicate. We establish a uniform bound
on the rescaled mass.

\begin{theorem}\label{theoremNonConcentrationSpecificInfinity}
Let \(N \ge 3\), \(p = 2\), \(V \in C(\R^N;[0, \infty))\). Assume that \(\alpha = N - 2\).
If \(V\) vanishes on a nonempty open set \(U\subset\R^N\),
then there exists a constant \(C > 0\) such that for every positive solution
\(u_\varepsilon\in H^1_\mathrm{loc}(\R^N)\cap L^2((1+\abs{x})^{-2}\dif x)\)
of \eqref{equationNLChoquard} it holds
\[
 \frac{1}{\varepsilon^N} \int_{\R^N} \abs{u_\varepsilon}^2 \le C.
\]
If
\[
 \lim_{R \to \infty} \frac{1}{R^{N - 2}} \int_{B_{2 R} \setminus B_R} V = 0
\]
then for every positive solution \(u_\varepsilon\in H^1_\mathrm{loc}(\R^N)\cap L^2((1+\abs{x})^{-2}\dif x)\) of \eqref{equationNLChoquard} it holds
\[
 \frac{1}{\varepsilon^N} \int_{\R^N} \abs{u_\varepsilon}^2 \le \Gamma(\tfrac{N - 2}{2})\pi^{N/2}2^{N - 2} \Bigl(\frac{N - 2}{2} \Bigr)^2.
\]
\end{theorem}

If \(V\) is compactly supported then
\[
 \limsup_{\abs{x} \to \infty} \abs{x}^{N-2-m}u_\varepsilon(x) > 0
\]
for some \(m=m (u_\varepsilon) > 0\) \cite{MVSNENLSNE}*{proposition~4.13}. In particular, in view of theorem~\ref{theoremNonConcentrationSpecificInfinity} this implies that \eqref{equationNLChoquard}
does not have positive solutions when \(N = 3, 4\).
We leave as an open problem whether theorem~\ref{theoremNonConcentrationSpecificInfinity} in combination with the above decay estimate brings an obstruction for the existence of concentrating positive solutions in other cases.

\subsection{The locally superlinear case $p>2$}
The same variational penalization technique in the case \(p > 2\) gives us the following existence and concentration result.

\begin{theorem}
\label{theoremSuperlinear}Let \(N \in \N\), \(\alpha \in ((N-4)_+, N)\), \(p \in (2, \frac{N + \alpha}{(N - 2)_+})\) and \(V \in C(\R^N;[0, \infty))\).
Assume that either
\[
  p > 1 + \frac{\max(\alpha, 1 + \frac{\alpha}{2})}{(N - 2)_+}
\]
or
\[
  \liminf_{\abs{x} \to \infty} V (x) \abs{x}^{2} > 0.
\]
If \(\Lambda \subset \R^N\) is an open bounded set such that
\[
  0 < \inf_{\Lambda} V < \inf_{\partial \Lambda} V,
\]
then problem \eqref{equationNLChoquard} has a family of positive solutions \((u_\varepsilon)_{\varepsilon \in (0, \varepsilon_0)}
\in H^1_{V}(\R^N)\) such that for a family of points \((a_\varepsilon)_{\varepsilon \in (0, \varepsilon_0)}\) in \(\Lambda\) and for every \(\rho > 0\),
\begin{align*}
  \lim_{\varepsilon \to 0} V (a_\varepsilon) & = \inf_{\Lambda} V,&
  \liminf_{\varepsilon \to 0} \varepsilon^{-N} \int_{B_{\varepsilon \rho} (a_\varepsilon)} \abs{u_\varepsilon}^2 &> 0,&
  \lim_{\substack{R \to \infty\\ \varepsilon \to 0}} \norm{u_\varepsilon}_{L^{\infty} (\R^N \setminus B_{\varepsilon R} (a_\varepsilon))}  &= 0.
\end{align*}
\end{theorem}

When \(N \le 2\), the condition on the potential \(V\) is always
\[
  \liminf_{\abs{x} \to \infty} V (x) \abs{x}^2 > 0.
\]

The assumptions of theorem \ref{theoremSuperlinear} are again optimal in the following sense.
First, the optimality of the restriction \(p < \frac{N + \alpha}{(N - 2)_+}\) is justified
by the following result for the limiting problem \eqref{eqLimit}.

\begin{theorem}
\label{theoremExistenceLimitSuperlinear}
Let \(N\in\N\) and \(\alpha\in(0,N)\).
Then limiting equation \eqref{eqLimit} has a nontrivial solution \(v_0 \in H^1(\R^N)\cap C^2 (\R^N)\)
if and only if \(p\in\big(\frac{N + \alpha}{N},\frac{N + \alpha}{(N - 2)_+}\big)\).
\end{theorem}

The proof of the existence of a positive solution can be found in \cite{MVSGNLSNE} (see also \citelist{\cite{GenevVenkov2012}\cite{Ma-Zhao-2010}}). The necessity of the restrictions on \(p\)
follows from the adapted Poho\v zaev identity \citelist{\cite{MVSGNLSNE}*{theorem 2}\cite{MVSGGCE}*{proposition~3.5}\cite{Menzala1983}\cite{CingolaniSecchiSquassina2010}*{lemma 2.1}}.

The additional restrictions on the decay of \(V\) are optimal in view of the following Liouville theorem for \eqref{equationNLChoquard}.
\begin{theorem}
[Van Schaftingen and Moroz \cite{MVSNENLSNE}*{theorems 1 and 2 and proposition~4.4}]
\label{theoremLiouvilleSuperLinear}
Let \(N \in N\), \(\alpha \in (0, N)\), \(p > 2\), \(V \in C(\R^N;[0, \infty))\) and \(\varepsilon > 0\).
If \eqref{equationNLChoquard} admits a positive distributional supersolution in \(\R^N\) and
\[
  p \le 1 + \frac{\max(\alpha, 1 + \frac{\alpha}{2})}{(N - 2)_+}
\]
then for every \(\gamma > 2\),
\[
  \limsup_{\abs{x} \to \infty} V (x) \abs{x}^\gamma > 0.
\]
\end{theorem}

We observe that when \(p > 2\), there is no strong critical potential well phenomenon
similar to theorems~\ref{theoremNonConcentrationSpecific} or \ref{theoremNonConcentrationSpecificInfinity}.

\subsection{Discussion of the penalization method}\label{sect-penalization}

In order to construct solutions we develop in this work a penalization method for the Cho\-quard equation which is inspired by the penalization method for the local nonlinear Schr\"odinger equation \eqref{equationLocal},
introduced by M.\thinspace del Pino and P.\thinspace Felmer \citelist{\cite{delPinoFelmer1997}\cite{delPinoFelmer1998}}, and by its adaptations to fast-decaying potentials \citelist{\cite{BonheureVanSchaftingen2006}\cite{BonheureVanSchaftingen2008}\cite{BonheureDiCosmoVanSchaftingen2012}\cite{DiCosmoVanSchaftingen}\cite{MorozVanSchaftingen2009}\cite{MorozVanSchaftingen2010}\cite{YinZhang2009}}.

Essentially, the penalization method for the local nonlinear Schr\"o\-dinger equation \eqref{equationLocal} consists in modifying the nonlinearity for large \(u\) and \(x\) in such a way that the modified variational problem becomes well-posed and solutions can be constructed by a standard variational argument \cite{delPinoFelmer1997}.
One has then to show that solutions of the penalized problem are small enough for large \(x\), so that they also solve the original problem. This is usually done by constructing supersolutions to a linearization of the penalized problem in an outer domain and by using them to estimate the solutions of the penalized problem by some comparison principle.
This penalization scheme should not be confused with the global penalization scheme of J.\thinspace Byeon and Z.-Q.\thinspace Wang which penalizes globally the \(L^2\) norm outside a concentration \citelist{\cite{ByeonWang2003}\cite{CingolaniSecchiSquassina2010}\cite{CingolaniJeanjeanSecchi2009}}.

Besides the technical issue of adapting estimates from the nonlinear Schr\"o\-dinger equation to the nonlocal Choquard equation, the essential difficulty that we faced in the study of \eqref{equationNLChoquard} by a penalization method was to find a natural way of cutting off the nonlinearity in order to improve the compactness properties of the functional while remaining able to prove that when \(\varepsilon > 0\) is small, solutions of the penalized problem solve the original equation \eqref{equationNLChoquard}.
Our penalization scheme depends on a penalization potential that should be defined adequately for each problem.
Whereas for the local nonlinear Schr\"odinger equation, the penalization potential was usually chosen as the largest potential that could be absorbed by the linear part of the operator \cite{MorozVanSchaftingen2010}, for the nonlocal Choquard equation taking a large potential might obstruct the construction of supersolution of the penalized problem.
We introduce a novel approach that works in the opposite direction.
We first construct supersolutions of the original equation \eqref{equationNLChoquard} in outer domains and then derive from them a penalization potential. Such an approach automatically gives supersolutions to the linearized problem, while checking that the supersolutions give an admissible penalization potential and is easier than constructing supersolutions for a given penalization potential. This approach should help to design penalization schemes
for large classes of problems and links the ideas of constructing solutions via supersolutions,
and by penalization schemes.

\section{The penalized problem}
In this section we define the penalized problem and functional and prove the existence of solutions to the latter.

\subsection{Function spaces and inequalities.}\label{sect-H-V-definition}
The linear part of the equation \eqref{equationNLChoquard} with a nonzero potential
\(V \in C (\R^N; [0, \infty))\) induces the norm
\[
\norm{u}_\varepsilon^2:=\int_{\R^N}\left(\varepsilon^2\abs{\nabla u}^2+V\abs{u}^2\right).
\]
For a nonempty open set \(\Omega\subseteq\R^N\), we denote by \(H^1_{V,0}(\Omega)\) the Hilbert space obtained by completion of the set of smooth test functions \(C^\infty_c(\Omega)\) with respect to the norm \(\norm{\cdot}_\varepsilon\). The completion is independent of \(\varepsilon > 0\). If \(\Omega=\R^N\) we simply write \(H^1_{V}(\R^N)\).

For \(N\ge 3\) the Sobolev inequality implies that \(H^1_{V,0}(\Omega)\subset L^\frac{2N}{N-2}(\Omega)\).
If \(N \le 2\) and \(V=0\) then the space \(H^1_{V}(\R^N)\) cannot be identified as a space of measurable functions or a space of distributions (see  \cite{PinchoverTintarev2006}*{Section 1} for a discussion).
However, if \(N \le 2\) and \(\inf_U V>0\) for an open set \(U\subset\R^N\) then
the space \(H^1_{V}(\R^N)\) is continuously embedded into \(L^2(\R^N, H (x)\dif x)\) with a suitable weight \(H\in C (\R^N; (0, \infty))\),
see \cite{MorozVanSchaftingen2010}*{Section 6.1} for the details in the case \(N=2\) (the case \(N=1\) can be treated similarly).

We shall review several functional inequalities on this space which are used in the sequel. The first inequality is a rescaled Sobolev inequality.

\begin{proposition}[Rescaled Sobolev inequality]
\label{propositionSobolevScaling}
Assume that either \(N \ge 3\) and \( \frac{1}{2} - \frac{1}{N} \le \frac{1}{q} \le \frac{1}{2}\) or \(N \in \{1,2\}\) and \(0<\frac{1}{q} \le \frac{1}{2}\).
If \(\Lambda\subset\R^N\) is a bounded open set and \(\inf_{\Lambda} V > 0\), then for every \(\varphi \in H^1_{V}(\R^N)\),
\[
 \int_{\Lambda} \abs{\varphi}^q \le \frac{C}{\varepsilon^{N (\frac{q}{2} - 1)}} \Bigl(\int_{\R^N} \varepsilon^2 \abs{\nabla \varphi}^2 + V \abs{\varphi}^2 \Bigr)^\frac{q}{2},
\]
where \(C>0\) depends only on \(\alpha\), \(N\) and \(\Lambda\).
\end{proposition}

When \(\liminf_{\abs{x} \to \infty} V (x) \abs{x}^2 > 0\), proposition~\ref{propositionSobolevScaling} is a simple case of a general inequality by
A.\thinspace Ambrosetti, V.\thinspace Felli and A.\thinspace Malchiodi \cite{AmbrosettiFelliMalchiodi2005}*{proposition~7};
we give a short proof in our setting.
\begin{proof}[Proof of proposition~\ref{propositionSobolevScaling}]
We take a cut-off function \(\eta \in C^\infty_c (\R^N)\) such that \(\eta = 1\) on \(\Bar{\Lambda}\) and \(\inf_{\supp \eta} V > 0\).
By the classical Sobolev inequality, we estimate
\[
\begin{split}
\int_{\Lambda} \abs{\varphi}^q \le \int_{\R^N} \abs{\eta \varphi}^q&\le  C\varepsilon^{N(1-\frac{2}{q})}\Big(\int_{\R^N} \big(\varepsilon^2\abs{\nabla (\eta \varphi)}^2 + \abs{\eta \varphi}^2\big)\Big)^\frac{q}{2}\\
& \le C'C\varepsilon^{N(1-\frac{2}{q})}  \Bigl(\int_{\R^N} \varepsilon^2 \abs{\nabla \varphi}^2 + V \abs{\varphi}^2 \Bigr)^\frac{q}{2}.\qedhere
\end{split}
\]
\end{proof}

In order to control the nonlocal term in the problem \eqref{equationNLChoquard} we will use the classical
Hardy--Littlewood--Sobolev inequality.

\begin{proposition}[Hardy--Littlewood--Sobolev inequality, \cite{LiebLoss2001}*{theorem 4.3}]
\label{propositionHLS}
Let \(N \in \N\), \(\alpha \in (0, N)\) and \(s \in (1, \frac{N}{\alpha})\).
If \(\varphi \in L^s (\R^N)\), then \(I_\alpha \ast \varphi\in L^\frac{N s}{N - \alpha s} (\R^N)\) and
\begin{equation*}
 \int_{\R^N} \abs{I_\alpha \ast \varphi}^\frac{N s}{N - \alpha s} \le C \Bigl(\int_{\R^N} \abs{\varphi}^s \Bigr)^\frac{N}{N - \alpha s},
\end{equation*}
where \(C>0\) depends only on \(\alpha\), \(N\) and \(s\).
\end{proposition}

We shall also rely on a weighted Hardy--Littlewood--Sobolev inequality.

\begin{proposition}[Weighted Hardy--Littlewood--Sobolev inequality, \cite{Stein-Weiss}]
\label{propositionHLSW}
Let \(N \in \N\) and \(\alpha \in (0, N)\).
If \(\varphi \in L^2 (\R^N,\abs{x}^\alpha \dif x)\), then \(I_\frac{\alpha}{2} \ast \varphi\in L^2(\R^N)\) and
\begin{equation*}
\label{eqHLS-SW}
\int_{\R^N}\bigabs{I_\frac{\alpha}{2}\ast \varphi}^2\le
C_\alpha\int_{\R^N}\abs{\varphi(x)}^2\abs{x}^\alpha \dif x,
\end{equation*}
where \(C_\alpha=\frac{1}{2^\alpha} \Bigl(\frac{\Gamma (\frac{N - \alpha}{4})}{\Gamma (\frac{N + \alpha}{4})} \Bigr)^2\).
\end{proposition}

This optimal constant \(C_\alpha\) was computed by I.\thinspace Herbst \cite{He1977}*{theorem 2.5}.
See also \cite{MVSHardy} for a simplified proof and further references.

\subsection{Definition of the penalized functional}
In what follows we fix the potential \(V \in C (\R^N; [0, \infty))\)
and a bounded nonempty open set \(\Lambda \subset \R^N\) such that
\[
  \inf_{\partial\Lambda} V>\inf_{\Lambda} V >0.
\]
Without loss of generality we can assume that the boundary of \(\Lambda\) is smooth, and that \(0\in\Lambda\).

We choose a family of \emph{penalization potentials} \(H_\varepsilon \in L^\infty(\R^N,[0, \infty))\) for \(\varepsilon > 0\) in such a way that
\(H_\varepsilon(x)=0\) for all \(x\in\Lambda\), and
\begin{equation}\label{eqHtozero}
\lim_{\varepsilon\to 0}\sup_{\R^N\setminus\Lambda} H_\varepsilon = 0.
\end{equation}
The explicit construction of \(H_\varepsilon\) will be described later, in section~\ref{Section-Back}.
Before that, we shall only rely on the following two assumptions on \(H_\varepsilon\):

\begin{enumerate}
  \item[\((H_1)\)] the space \(H^1_{V}(\R^N)\) is compactly embedded in
 \(L^2(\R^N, \bigl(H_\varepsilon (x)^2\abs{x}^\alpha+\chi_{\Lambda}(x)) \dif x\bigr)\),
 \item[\((H_2)\)] there exists \(\kappa > 0\) such that \(C_\alpha p\kappa < 1\) and for every \(\varphi \in H^1_{V} (\R^N)\),
 \[
    \frac{1}{\varepsilon^\alpha}\int_{\R^N} \abs{H_\varepsilon (x)\varphi (x)}^2 \abs{x}^\alpha \dif x\le \kappa \int_{\R^N} \varepsilon^2 \abs{\nabla \varphi}^2 + V \abs{\varphi}^2,
 \]
 where \(C_\alpha>0\) is the optimal constant in the weighted Hardy--Littlewood--Sobolev inequality of proposition~\ref{propositionHLSW}.
\end{enumerate}
The condition \((H_2)\) can be seen as a Hardy type inequality with a constant bounded uniformly with respect to \(\varepsilon\).
The assumption \((H_2)\) via proposition~\ref{propositionHLSW} implies that the associated to \(H_\varepsilon\) nonlocal quadratic form is well--defined on the space \(H^1_{V} (\R^N)\) and for every \(\varphi \in H^1_{V} (\R^N)\),
\begin{equation*}
\frac{p}{\varepsilon^\alpha}\int_{\R^N}\bigabs{I_\frac{\alpha}{2}\ast (H_\varepsilon\varphi)}^2\le
(1-\delta)\int_{\R^N} \varepsilon^2 \abs{\nabla \varphi}^2 + V \abs{\varphi}^2,
\end{equation*}
where \(\delta=1-C_\alpha p\kappa>0\).

Given a penalization potential \(H_\varepsilon\) which satisfies \((H_1)\) and \((H_2)\), we define the \emph{penalized nonlinearity}
\(g_\varepsilon:\R^N\times\R\to \R\) for \(x\in\R^N\) and \(s\in\R\) by
\[
 g_\varepsilon (x, s): = \chi_\Lambda(x) s_+^{p-1}+\chi_{\R^N\setminus\Lambda}(x)\min\big( s_+^{p-1},  H_\varepsilon (x)\big).
\]
We also denote \(G_\varepsilon(x,s)=\int_0^s g_\varepsilon(x,t)\dif t\).
The function \(g_\varepsilon\) is a Carath\'eodory function
which satisfy the following properties:
\begin{enumerate}
\item[\((g_1)\)]\label{assumptGBoundp} for every \(x \in \R^N\) and \(s \in \R\), \(g_\varepsilon(x,s)\le s_+^{p - 1}\),
\item[\((g_2)\)] for every \(s\in\R\) and \(x\in\R^N\),
\[
0\le g_\varepsilon(x,s)s\le s_+^p\chi_\Lambda + H_\varepsilon(x)s_+(1-\chi_\Lambda),
\]
\[0\le G_\varepsilon(x,s)\le \frac{1}{p}s_+^p\chi_\Lambda+H_\varepsilon(x)s_+(1-\chi_\Lambda),\]
\item[\((g_3)\)] for every \(s\in\R\) and \(x \in \R^N\),
\[
  0\le G_\varepsilon(x,s)\le g_\varepsilon(x,s)s,
\]
and for every \(x\in\Lambda\),
\[
  0\le pG_\varepsilon(x,s)= g_\varepsilon(x,s)s.
\]
\end{enumerate}
We define the \emph{penalized superposition operators} \(\mathscr{g}_\varepsilon\) and \(\mathscr{G}_\varepsilon\) for \(u : \R^N \to \R\) and \(x \in \R^N\) by
\begin{align*}
 \mathscr{g}_\varepsilon (u) (x)&= g_\varepsilon \bigl(x, u (x)\bigr)  &
 & \text{ and } &
 \mathscr{G}_\varepsilon (u) (x)&= G_\varepsilon \bigl(x, u (x)\bigr),
\end{align*}
and the \emph{penalized functional} \(\J_\varepsilon:H^1_{V}(\R^N)\to\R\) for \(u \in H^1_{V} (\R^N)\) by
\[
\J_\varepsilon(u)= \frac{1}{2}\int_{\R^N} \bigl( \varepsilon^2\abs{\nabla u}^2 + V\abs{u}^2\bigr)
 - \frac{p}{2 \varepsilon^\alpha} \int_{\R^N} \bigabs{I_\frac{\alpha}{2}\ast \mathscr{G}_\varepsilon (u)}^2.
\]

\begin{lemma}[Elementary properties of the penalized functional]
\label{lemmaFunctional}
If \(1<p<\frac{N+\alpha}{(N-2)_+}\) and the assumption \((H_1)\) holds,
then \(\J_\varepsilon\in C^1(H^1_{V}(\R^N),\R)\) and for every \(u \in H^1_{V} (\R^N)\) and \(\varphi\in H^1_{V}(\R^N)\),
\begin{equation*}
\dualprod{\J_\varepsilon^\prime(u)}{\varphi} = \int_{\R^N} \bigl(\varepsilon^2\scalprod{\nabla u}{\nabla \varphi} + Vu\varphi\bigr)
- \frac{p}{\varepsilon^\alpha} \int_{\R^N} \bigl(I_\alpha\ast \mathscr{G}_\varepsilon (u)\bigr) \mathscr{g}_\varepsilon (u)\varphi.
\end{equation*}
\end{lemma}

Here \(\dualprod{\cdot}{\cdot}\) denotes the duality product between the dual space \(H^1_{V} (\R^N)'\) and the space \(H^1_{V} (\R^N)\). In particular, \(u\in H^1_{V}(\R^N)\) is a critical point of \(\J_\varepsilon\) if and only if \(u\) is a weak solution of the \emph{penalized equation}
\begin{equation}
\label{equationPenalized}\tag{$\mathcal{Q_\varepsilon}$}
- \varepsilon^2\Delta u + Vu = p\varepsilon^{-\alpha}\bigl(I_\alpha \ast \mathscr{G}_\varepsilon (u)\bigr) \mathscr{g}_\varepsilon (u) \quad \text{in \(\R^N\)}.
\end{equation}

\begin{proof}[Proof of lemma~\ref{lemmaFunctional}]
We only need to consider the nonlocal part of \(\J_\varepsilon\),
\[
  \mathcal{N}_\varepsilon(u)= \int_{\R^N} \bigabs{I_\frac{\alpha}{2}\ast \mathscr{G}_\varepsilon(u)}^2.
\]
Since \(\frac{(N-2)_+}{N+\alpha}<\frac{1}{p}\), the classical Sobolev embedding and the assumption \((H_1)\) imply that
the space \(H^1_{V}(\R^N)\) is continuously embedded into
the direct sum
\(L^2(\R^N\setminus\Lambda,H_\varepsilon (x)^2\abs{x}^\alpha\dif x)\oplus L^{2Np/(N+\alpha)}(\Lambda)\) which is naturally isomorphic to a Banach spaces of measurable functions on \(\R^N\), it is sufficient to prove the continuous Fr\'echet differentiability on the latter space.

By the growth assumption \((g_2)\) and standard continuity properties
of superposition operators (see for example \citelist{\cite{Rabinowitz1986}*{proposition~B.1}\cite{WillemMinimax}*{Theorem A.4}\cite{Appell-Zabreiko}*{theorem 3.1}\cite{AmbrosettiProdi1993}*{theorem 2.2}}), the nonlinear superposition operator
\[\mathscr{G}_\varepsilon:L^2(\R^N\setminus\Lambda,H_\varepsilon (x)^2\abs{x}^\alpha\dif x)\oplus L^\frac{2Np}{N+\alpha}(\Lambda) \to
L^2(\R^N\setminus\Lambda,\abs{x}^\alpha\dif x)\oplus L^\frac{2N}{N+\alpha}(\Lambda)\]
is continuous.
In view of the classical and weighted Hardy--Littlewood--Sobolev inequalities of propositions~\ref{propositionHLS} and \ref{propositionHLSW}, the Riesz potential integral operator
\[
  f \in L^2(\R^N\setminus\Lambda,\abs{x}^\alpha\dif x)\oplus L^\frac{2N}{N+\alpha}(\Lambda)\mapsto
  I_\frac{\alpha}{2} \ast f \in L^2(\R^N)
\]
is a continuous linear operator. We deduce therefrom that \(\mathcal{N}_\varepsilon:H^1_{V}(\R^N)\mapsto\R\) is continuous by composition.

For the differentiability, we observe first that the map \(\mathcal{N}_\varepsilon\) is G\^ateaux--differentiable on the space \(L^2(\R^N\setminus\Lambda,H_\varepsilon (x)^2\abs{x}^\alpha\dif x)\oplus L^{2Np/(N+\alpha)}(\Lambda)\) and that its G\^ateaux derivative \(\mathcal{N}_\varepsilon'\) satisfies for every \(u, \varphi \in L^2(\R^N\setminus\Lambda,H_\varepsilon (x)^2\abs{x}^\alpha\dif x)\oplus L^{2Np/(N+\alpha)}(\Lambda)\),
\[
 \dualprod{\mathcal{N}_\varepsilon' (u)}{\varphi} = \int_{\R^N} \bigl(I_\alpha \ast \mathscr{G}_\varepsilon (u)\bigr) \mathscr{g}_\varepsilon (u) \varphi.
\]
By our observations on \(I_{\alpha/2}\) above and by duality, the convolution with \(I_\alpha = I_{\alpha/2} \ast I_{\alpha/2}\) is a bounded linear operator from \(L^2(\R^N\setminus\Lambda,\abs{x}^\alpha\dif x)\oplus L^{2N/(N+\alpha)}(\Lambda)\) to \(L^2(\R^N\setminus\Lambda,\abs{x}^{-\alpha}\dif x)\oplus L^{2N/(N-\alpha)}(\Lambda)\), and thus the nonlinear operator
\[
 u \in L^2(\R^N\setminus\Lambda,H_\varepsilon (x)^2 \abs{x}^\alpha\dif x)\oplus L^\frac{2N p}{N+\alpha}(\Lambda)
 \mapsto I_\alpha \ast \mathscr{G}_\varepsilon (u) \in L^2(\R^N\setminus\Lambda,\abs{x}^{-\alpha}\dif x)\oplus L^\frac{2N}{N-\alpha}(\Lambda)
\]
is continuous. By the growth assumption \((g_2)\) and by our hypothesis \((H_1)\) on the penalization potential, the nonlinear superposition operator \(\mathscr{g}_\varepsilon\) is continuous from \(L^{2Np/(N + \alpha)} (\Lambda)\) to \(L^{2Np/((N + \alpha)(p - 1))} (\Lambda)\). Since \(g_\varepsilon\) is a Carath\'eodory function, the superposition operator \(\mathscr{g}_\varepsilon\) is continuous from \(L^2(\R^N\setminus\Lambda,H_\varepsilon (x)^2 \abs{x}^\alpha\dif x)\) to the set of measurable functions on \(\R^N \setminus \Lambda\) with the topology of convergence in measure on finite-measure sets.
Since \(\abs{\mathscr{g}_\varepsilon (u)} \le H_\varepsilon\) in \(\R^N \setminus \Lambda\), we conclude by Lebesgue's dominated convergence theorem that the nonlinear operator
\begin{multline*}
 u \in L^2(\R^N\setminus\Lambda,\abs{x}^\alpha\dif x)\oplus L^\frac{2N p}{N+\alpha}(\Lambda)\\
 \mapsto (I_\alpha \ast \mathscr{G}_\varepsilon (u))\mathscr{g}_\varepsilon (u) \in L^2(\R^N\setminus\Lambda,H_\varepsilon (x)^{-2}\abs{x}^{-\alpha}\dif x)\oplus L^\frac{2N p}{(2 N - 1)p-\alpha}(\Lambda)
\end{multline*}
is continuous. The map \(\mathcal{N}_\varepsilon\) is continuously G\^ateaux differentiable on \( L^2(\R^N\setminus\Lambda,H_\varepsilon (x)^2\abs{x}^\alpha\dif x)\oplus L^{2Np/(N+\alpha)}(\Lambda)\). Hence it is continuously Fr\'echet differentiable on that space \citelist{\cite{Schwartz1969}*{lemma 1.15}\cite{WillemMinimax}*{proposition~1.3}\cite{AmbrosettiProdi1993}*{theorem 1.9}} and the conclusion follows.
\end{proof}

The difficulty in this proof is that the superposition operator
\[
\mathscr{G}_\varepsilon:L^2(\R^N \setminus \Lambda,H_\varepsilon (x)^2\abs{x}^\alpha\dif x) \to
L^2(\R^N\setminus\Lambda,\abs{x}^\alpha\dif x)
\]
is not Fr\'echet differentiable \cite{Appell-Zabreiko}*{section 3.6}.

\subsection{Existence of solutions of the penalized problem}
We now construct solutions to the penalized problem \eqref{equationPenalized} as critical points of the penalized functional \(\J_\varepsilon\).

\begin{proposition}[Existence of solutions of the penalized problem]\label{propositionPenalizedExistence}
If \(1<p<\frac{N+\alpha}{(N-2)_+}\) and the assumptions \((H_1)\) and \((H_2)\) hold,
then problem \eqref{equationPenalized} has a nonnegative solution \(u_\varepsilon \in H^1_{V} (\R^N)\).
Moreover,
\[
  \J_\varepsilon(u_\varepsilon)=c_\varepsilon :=\inf_{\gamma\in\Gamma_\varepsilon}\max_{t\in[0, 1]}\J_\varepsilon\bigl(\gamma(t)\bigr)>0,
\]
where \(\Gamma_\varepsilon:=\bigl\{\gamma\in C([0, 1], H^1_{V}(\R^N))\mid \gamma(0)=0, \J_\varepsilon\bigl(\gamma(1)\bigr)<0\bigr\}\).
\end{proposition}

We call every critical point \(u \in H^1_{V} (\R^N)\) of \(\J_\varepsilon\) such that \(\J_\varepsilon(u)=c_\varepsilon\)
a \emph{ground state solution} of problem \eqref{equationPenalized}.

In order to prove proposition~\ref{propositionPenalizedExistence}
in the rest of this section, we will show that the penalized functional \(\J_\varepsilon\) has a mountain--pass geometry
and satisfies the Palais--Smale condition,
thus fulfilling the assumptions of the mountain-pass lemma \citelist{\cite{AmbrosettiRabinowitz1973}*{theorem 2.1}\cite{Rabinowitz1986}*{theorem 2.2}\cite{Struwe2008}*{theorem 6.1}\cite{WillemMinimax}*{theorem 2.10}}.

\begin{lemma}[Mountain--pass geometry]
\label{lemmaMPL}
Assume that \(1<p<\frac{N+\alpha}{(N-2)_+}\) and \((H_2)\) holds.
Then the functional \(\J_\varepsilon\) is unbounded from below and \(u=0\) is a strict local minimum of \(\J_\varepsilon\).
\end{lemma}

\begin{proof}
To see that the functional \(\J_\varepsilon\) is unbounded from below, we choose \(\varphi\in C^\infty_c(\Lambda,[0,+\infty)) \setminus \{0\}\) and we observe that \(\J_\varepsilon(\tau\varphi)\to-\infty\) as \(\tau\to\infty\).

Next we show that \(0\) is a strict local minimum of \(\J_\varepsilon\).
Let \(u \in H^1_V (\R^N)\). By the growth assumption \((g_2)\), we have in \(\R^N\),
\[
  \bigabs{I_\frac{\alpha}{2}\ast \mathscr{G}_\varepsilon(u)}
  \le \frac{1}{p} \bigabs{I_\frac{\alpha}{2}\ast (\chi_\Lambda u_+^p)} + \abs{I_\frac{\alpha}{2}\ast (H_\varepsilon u_+)}.
\]
In view of the Young inequality, we obtain for any \(\lambda>0\)
\begin{equation}\label{eq-I-eps-bound}
\frac{p}{2 \varepsilon^\alpha} \int_{\R^N} \bigabs{I_\frac{\alpha}{2}\ast \mathscr{G}_\varepsilon(u)}^2
\le \frac{1+\lambda}{2p\varepsilon^\alpha}\int_{\R^N}\bigabs{I_\frac{\alpha}{2}\ast (\chi_\Lambda u_+^p)}^2+
\frac{p}{2\varepsilon^\alpha}\Big(1+\frac{1}{\lambda}\Big)\int_{\R^N}\bigabs{I_\frac{\alpha}{2}\ast (H_\varepsilon u_+)}^2.
\end{equation}
Since \(1< p\le\frac{N+\alpha}{N-2}\), using the classical Hardy--Littlewood--Sobolev (proposition~\ref{propositionHLS})
and the rescaled Sobolev inequality (proposition~\ref{propositionSobolevScaling}), we obtain
\begin{equation*}
\frac{p}{2 \varepsilon^\alpha} \int_{\R^N}\bigabs{I_\frac{\alpha}{2}\ast (\chi_\Lambda u_+^p)}^2
\le \frac{C}{\varepsilon^\alpha} \left(\int_\Lambda \abs{u}^\frac{2Np}{N+\alpha}\right)^\frac{N+\alpha}{N}\le \frac{C'}{\varepsilon^{(p - 1) N}} \Big(\int_{\R^N} \varepsilon^2 \abs{\nabla u}^2 + V \abs{u}^2\Big)^p.
\end{equation*}
Using the weighted Hardy--Littlewood--Sobolev inequality (proposition~\ref{propositionHLSW}) and the assumption \((H_2)\), we have
\begin{equation*}
\frac{p}{\varepsilon^\alpha}\int_{\R^N} \abs{I_\frac{\alpha}{2}\ast H_\varepsilon u_+}^2 \le
{C_\alpha p\kappa}\int_{\R^N} \big(\varepsilon^2\abs{\nabla u}^2 + V\abs{u}^2\big),
\end{equation*}
where \(C_\alpha p\kappa<1\).
We now choose \(\lambda>0\)
in the bound \eqref{eq-I-eps-bound} so that \(C_\alpha p\kappa\big(1+\frac{1}{\lambda}\big)<1\).
Then from \eqref{eq-I-eps-bound} we obtain
\begin{multline*}
\J_\varepsilon(u)\ge \frac{1}{2}\int_{\R^N} \Big( \varepsilon^2\abs{\nabla u}^2 + V\abs{u}^2\Big)\\
-\frac{p}{2\varepsilon^\alpha}\Big(1+\frac{1}{\lambda}\Big)\int_{\R^N} \bigabs{I_\frac{\alpha}{2}\ast (H_\varepsilon u_+)}^2
-\frac{1+\lambda}{2p\varepsilon^\alpha}\int_{\R^N}\bigabs{I_\frac{\alpha}{2}\ast (\chi_\Lambda u_+^p)}^2\\
\ge\frac{1}{2}\Big(1-C_\alpha p\kappa\big(1+\frac{1}{\lambda}\big)\Big)\int_{\R^N} \big( \varepsilon^2\abs{\nabla u}^2 + V\abs{u}^2\big)-\frac{C''}{\varepsilon^{(p - 1) N}}\Bigl(\int_{\R^N}\varepsilon^2 \abs{\nabla u}^{2} + V \abs{u}^2 \Bigr)^p,
\end{multline*}
the assertion follows since \(p>1\).
\end{proof}

\begin{lemma}[Coerciveness property]
\label{lemmaCoerciveness}
Assume that \(1<p<\frac{N+\alpha}{(N-2)_+}\) and \((H_2)\) holds. For every \(\kappa < \frac{1}{C_\alpha p}\),
there exists \(C_{\kappa, p}>0\) and \(\lambda_{\kappa, p}>0\) such that for every \(u \in H^1_{V} (\R^N)\).
\[
\int_{\R^N} \bigl(\varepsilon^2 \abs{\nabla u}^2 + V \abs{u}^2\bigr) \le C_{\kappa, p} \J_\varepsilon (u)+ \lambda_{\kappa, p} \dualprod{\J_\varepsilon^\prime(u)}{u}.
\]
\end{lemma}

Note that the constant does not depend on the potential \(V\), the set \(\Lambda\) or the penalization potential \(H_\varepsilon\) except via \(\kappa\).
This lemma will imply in particular the boundedness of Palais-Smale sequences.

\begin{proof}[Proof of lemma~\ref{lemmaCoerciveness}]
Given \(\theta\) such that \(\frac{1}{2p}<\theta<\frac{1}{2}\), we shall estimate
\begin{multline*}
J := \Big(\frac{1}{2}-\theta\Big)\int_{\R^N} \big(\varepsilon^2 \abs{\nabla u}^2 + \abs{u}^2\big) -
\J_\varepsilon(u)+\theta\dualprod{\J^\prime_\varepsilon(u)}{u}\\
=
\frac{p}{2 \varepsilon^\alpha} \int_{\R^N} \bigl(I_\alpha \ast \mathscr{G}_\varepsilon (u)\bigr) \bigl(\mathscr{G}_\varepsilon(u) - 2\theta \mathscr{g}_\varepsilon (u) u \bigr).
\end{multline*}
If \(\theta \ge \frac{1}{2 p}\), we obtain, as \(\mathscr{G}_\varepsilon(u)-2\theta \mathscr{g}_\varepsilon (u) u = \bigl(\frac{1}{p} - 2 \theta \bigr) u_+^p\) on \(\Lambda\),
\begin{equation*}
  \frac{p}{2} \int_{\Lambda} \bigl(I_\alpha \ast \mathscr{G}_\varepsilon (u)\bigr) \bigl(\mathscr{G}_\varepsilon (u) (y)-2\theta \mathscr{g}_\varepsilon (u) (y)u\bigr) \le -
  \Bigl(\theta - \frac{1}{2p}\Bigr)\int_{\Lambda} \bigl(I_\alpha \ast (\chi_{\Lambda} u_+^p)\bigr) u_+^p.
\end{equation*}
On the other hand, if \(\theta \le \frac{1}{2}\) then \(\mathscr{G}_\varepsilon(u) - 2\theta \mathscr{g}_\varepsilon (u) u \le (1 - 2 \theta) H_\varepsilon u_+\) on \(\R^N \setminus \Lambda\),
and therefore,
\begin{multline*}
 \frac{p}{2} \int_{\R^N \setminus \Lambda} \bigl(I_\alpha \ast \mathscr{G}_\varepsilon (u)\bigr) \bigl(\mathscr{G}_\varepsilon(u) - 2\theta \mathscr{g}_\varepsilon (u)u\bigr)\\
 \shoveleft{\qquad\qquad\le
 p\Bigl(\frac{1}{2} - \theta\Bigr) \int_{\R^N \setminus \Lambda} \bigl(I_\alpha \ast \mathscr{G}_\varepsilon (u)\bigr) H_\varepsilon u_+}\\
 \le \Bigl(\frac{1}{2} - \theta\Bigr) \int_{\R^N \setminus \Lambda} \bigl(I_\alpha \ast (\chi_\Lambda u_+^p) \bigr) H_\varepsilon u_+
 + p\Bigl(\frac{1}{2} - \theta\Bigr) \int_{\R^N \setminus \Lambda} \bigl(I_\alpha \ast H_\varepsilon u_+\bigr) H_\varepsilon u_+ \bigr.
\end{multline*}
By the Cauchy--Schwarz inequality, we deduce that
\begin{multline*}
 J \le - \Bigl(\theta - \frac{1}{2 p}\Bigr) \int_{\R^N} \bigabs{I_{\frac{\alpha}{2}} \ast (\chi_\Lambda u_+^p)}^2\\
 + \Bigl(\frac{1}{2} - \theta\Bigr) \Bigl(\int_{\R^N} \bigabs{I_{\frac{\alpha}{2}} \ast (\chi_\Lambda u_+^p)}^2\Bigr)^\frac{1}{2}\Bigl(\int_{\R^N} \bigabs{I_{\frac{\alpha}{2}} \ast (H_\varepsilon u_+)}^2\Bigr)^\frac{1}{2}\\
  + p\Bigl(\frac{1}{2} - \theta\Bigr)\int_{\R^N} \bigabs{I_{\frac{\alpha}{2}} \ast (H_\varepsilon u_+)}^2.
\end{multline*}
The right-hand side is quadratic in \(\bigl(\int_{\R^N} \bigabs{I_{\frac{\alpha}{2}} \ast (\chi_\Lambda u_+^p)}^2\bigr)^{1/2}\) and can thus be estimated to obtain the inequality
\[
 J \le \Bigl(\frac{1}{2} - \theta\Bigr)\Bigl(p + \frac{\frac{1}{2} - \theta}{4(\theta - \frac{1}{2p})}\Bigr)\int_{\R^N} \bigabs{I_{\frac{\alpha}{2}} \ast (H_\varepsilon u_+)}^2.
\]
In view of the assumption \((H_2)\), we have thus proved that
\[
\Big(\frac{1}{2}-\theta\Big)\biggl(1 - C_\alpha \kappa \Bigl(p + \frac{\frac{1}{2} - \theta}{4(\theta - \frac{1}{2p})}\Bigr)\biggr)\int_{\R^N} \bigl(\varepsilon^2 \abs{\nabla u}^2 + \abs{u}^2\bigr) \le
\J_\varepsilon (u)+\theta \dualprod{\J_\varepsilon^\prime(u)}{u}.
\]
We conclude by observing that if \(C_\alpha \kappa p < 1\), then when \(\theta\) is sufficiently close to \(\frac{1}{2}\),
\[
 \Big(\frac{1}{2}-\theta\Big)\biggl(1 - C_\alpha \kappa \Bigl(p + \frac{\frac{1}{2} - \theta}{4(\theta - \frac{1}{2p})}\Bigr)\biggr) > 0.\qedhere
\]
\end{proof}

\begin{lemma}[Palais-Smale condition]
\label{lemmaBPS}
Assume that \(1<p<\frac{N+\alpha}{(N-2)_+}\) and \((H_1)\), \((H_2)\) holds.
If \((u_n)\) is a Palais-Smale--sequence in \(H^1_{V}(\R^N)\) for the functional \(\J_\varepsilon\), that is, for some \(c \in \R\),
\[
\J_\varepsilon(u_n)\to c\quad \text{and} \quad \J_\varepsilon^\prime(u_n)\to 0.
\]
then, up to a subsequence, \((u_n)\) converges strongly to a limit \(u \in H^1_{V}(\R^N)\).
\end{lemma}

\begin{proof}
Lemma \ref{lemmaCoerciveness} implies in a standard way that the Palais--Smale sequence \((u_n)_{n \in \N}\) is bounded in \(H^1_{V}(\R^N)\). It follows that \(\dualprod{\J_\varepsilon^\prime(u_n)}{u_n}\to 0\) and that, up to a subsequence, \(u_n \weakto u \in H^1_{V}(\R^N)\) as \(n \to \infty\).
By our compactness assumption \((H_1)\) and by the Rellich--Kondrachov compactness theorem (see for example \citelist{\cite{LiebLoss2001}*{theorem 8.9}\cite{Adams1975}*{theorem 6.2}\cite{Willem2013}*{theorem 6.4.6}}) we conclude that, as \(n \to \infty\),
\[
  u_n \to u\quad\text{in \(L^2(\R^N\setminus\Lambda,H^2(x)\abs{x}^\alpha \dif x)\oplus L^\frac{2Np}{N+\alpha}(\Lambda)\)}.
\]
From the property \((g_2)\) of the nonlinearity it follows in a standard way that \(\mathscr{G}_\varepsilon\) as well as the superposition operator \(u\mapsto \mathscr{g}_\varepsilon (u) u\)
are bounded and continuous operators from \(L^2(\R^N\setminus\Lambda,H^2(x)\abs{x}^\alpha\dif x)\oplus L^\frac{2Np}{N+\alpha}(\Lambda)\) into
\(L^2(\R^N\setminus\Lambda,\abs{x}^\alpha\dif x)\oplus L^\frac{2N}{N+\alpha}(\Lambda)\) (see the proof of lemma~\ref{lemmaFunctional}), so we have, as \(n \to \infty\),
\[
\mathscr{G}_\varepsilon (u_n)\to \mathscr{G}_\varepsilon (u)\quad\text{and}\quad\mathscr{g}_\varepsilon (u_n)u_n\to \mathscr{g}_\varepsilon (u)u\quad\text{in \(L^2(\R^N\setminus\Lambda,\abs{x}^\alpha \dif x)\oplus L^\frac{2N}{N+\alpha}(\Lambda)\)}.
\]
By the classical and weighted Hardy--Littlewood--Sobolev inequalities of propositions~\ref{propositionHLS} and \ref{propositionHLSW} we conclude that, as \(n \to \infty\),
\[
  I_\frac{\alpha}{2}\ast\mathscr{G}_\varepsilon (u_n)\to I_\frac{\alpha}{2}\ast\mathscr{G}_\varepsilon (u)
\quad\text{and}\quad I_\frac{\alpha}{2}\ast(\mathscr{g}_\varepsilon (u_n)u_n)\to I_\frac{\alpha}{2}\ast(\mathscr{g}_\varepsilon (u)u)
\quad\text{in \(L^2(\R^N)\)}.
\]
Since \(u_n \weakto u \in H^1_{V}(\R^N)\) and \(\dualprod{\J^\prime(u_n)}{u_n}\to 0\) as \(n \to \infty\),
we estimate
\begin{multline*}
\limsup_{n \to \infty}\int_{\R^N} \bigl(\varepsilon^2 \abs{\nabla u_n-\nabla u}^2 + V \abs{u_n - u}^2\bigr) \\
\hspace{-15ex}= \limsup_{n \to \infty} \int_{\R^N} \bigl(\varepsilon^2 \abs{\nabla u_n} + V \abs{u_n}^2\bigr) -\int_{\R^N} \bigl(\abs{\nabla u}^2 + V \abs{u}^2\bigr) \\
\le\limsup_{n \to \infty} \frac{p}{\varepsilon^\alpha} \int_{\R^N} \bigl(I_\frac{\alpha}{2}\ast\mathscr{G}_\varepsilon (u_n)\bigr)\,\bigl(I_\frac{\alpha}{2}\ast\mathscr{g}_\varepsilon (u_n)u_n\bigr)\\
-\frac{p}{\varepsilon^\alpha}  \int_{\R^N} \bigl(I_\frac{\alpha}{2}\ast\mathscr{G}_\varepsilon (u)\bigr)\, \bigl(I_\frac{\alpha}{2}\ast\mathscr{g}_\varepsilon (u)u\bigr)=0,
\end{multline*}
which completes the proof.
\end{proof}

The same argument as above shows that the mapping \(u \mapsto \J_\varepsilon^\prime(u)\) is completely continuous from \(H^1_{V}(\R^N) \to \big(H^1_{V}(\R^N)\big)^*\), that is, \(\J_\varepsilon^\prime\) maps bounded sets to precompact sets.

\begin{proof}[Proof of proposition~\ref{propositionPenalizedExistence}]
Since \(\J_\varepsilon\) is continuously Fr\'echet differentiable (lemma~\ref{lemmaFunctional}), has the mountain--pass geometry (lemma~\ref{lemmaMPL}) and satisfies the Palais-Smale condition (lemma~\ref{lemmaBPS}),
all the assumptions of the mountain--pass lemma are satisfied \citelist{\cite{AmbrosettiRabinowitz1973}*{theorem 2.1}\cite{Rabinowitz1986}*{theorem 2.2}\cite{Struwe2008}*{theorem 6.1}\cite{WillemMinimax}*{theorem 2.10}}, and thus \(\J_\varepsilon\) admits
a critical point \(u_\varepsilon\in H^1_{V}(\R^N)\) at the critical level \(c_\varepsilon\).
Moreover, \(u_\varepsilon\in H^1_{V}(\R^N)\) is a solution of \eqref{equationPenalized}
Also, since \(-\varepsilon^2 +\Delta u_\varepsilon + Vu_\varepsilon\ge 0\) in \(\R^N\), by the weak maximum principle for weak solutions we have \(u_\varepsilon \ge 0\) in \(\R^N\).
\end{proof}

\section{Asymptotics of solutions of the penalized problem}

\subsection{The limiting problem}

For \(\lambda>0\), the \emph{limiting problem} associated to our problem \eqref{equationNLChoquard} is
\begin{equation}
\label{equationLimit}\tag{$\mathcal{R_\lambda}$}
- \Delta v + \lambda v = \bigl(I_\alpha \ast \abs{v}^p\bigr)\abs{v}^{p-2}v\quad \text{in \(\R^N\)},
\end{equation}
and the corresponding \emph{limiting functional} is \(\mathcal{I}_\lambda : H^1 (\R^N) \to \R\) defined for \(v \in H^1 (\R^N)\) by
\[
  \mathcal{I}_\lambda (v) = \frac{1}{2} \int_{\R^N} \abs{\nabla v}^2 + \lambda \abs{v}^2
- \frac{1}{2 p} \int_{\R^N} \bigabs{I_\frac{\alpha}{2} \ast \abs{v}^p}^2.
\]
By the Hardy--Littlewood--Sobolev inequality (proposition~\ref{propositionHLS}), the functional \(\I_\lambda\)
is well-defined on the space \(H^1(\R^N)\) for \(p \in [1, \infty)\) such that \(\frac{N+\alpha}{N}\le p\le \frac{N+\alpha}{(N-2)_+}\).
We define the \emph{limiting energy} by
\begin{equation}\label{Elambda}
 \mathcal{E} (\lambda) = \inf_{v \in H^1(\R^N)\setminus\{0\}} \max_{t > 0} \I_{\lambda} (t v).
\end{equation}
Since for every \(v \in H^1 (\R^N)\), \(\mathcal{I}_{\lambda} (\abs{v}) = \mathcal{I}_{\lambda} (v)\), the functional \(\mathcal{I}_\lambda\) is continuous and the set of compactly supported smooth functions \(C^\infty_c (\R^N)\) is dense in its domain \(H^1 (\R^N)\),
\begin{equation}\label{ElambdaCinfty}
 \mathcal{E} (\lambda) = \inf_{\substack{v\in C^\infty_c (\R^N) \setminus\{0\}\\ v \ge 0}} \max_{t > 0} \I_{\lambda} (t v).
\end{equation}

The following proposition provides information about the homogeneity of \(\I_\lambda\).

\begin{proposition}[Scaling of the limiting problem]
\label{propositionScaling}
Let \(\lambda > 0\) and \(v \in H^1 (\R^N)\). If \(v_\lambda \in H^1 (\R^N)\) is defined for every \(y \in \R^N\) by
\[
  v_\lambda (y) = \lambda^{\frac{\alpha + 2}{4 (p - 1)}} v (\sqrt{\lambda}\,y),
\]
then
\[
  \I_{\lambda} (v_\lambda) = \lambda^{\frac{\alpha + 2}{2(p - 1)} - \frac{N - 2}{2}} I_1 (v_1).
\]
In particular, \(v\) is a solution of \((\mathcal{R}_1)\) if and only if
\(v_\lambda\)
is a solution of \eqref{equationLimit}
and
\[
 \mathcal{E} (\lambda) = \mathcal{E} (1) \lambda^{\frac{\alpha + 2}{2(p - 1)} - \frac{N - 2}{2}}.
\]
\end{proposition}
In particular, if \(\mathcal{E}\) is continuous and, since \(p < \frac{N + \alpha}{N - 2}\), is increasing on \((0, \infty)\).
\begin{proof}[Proof of proposition~\ref{propositionScaling}]
The proof is a direct computation.
\end{proof}

The existence of solutions of Choquard equation \((\mathcal{R}_1)\) was proved for \(p = 2\) by variational methods by E.\thinspace H.\thinspace Lieb, P.-L.\thinspace Lions and G.\thinspace Menzala \citelist{\cite{Lieb1977}\cite{Lions1980}\cite{Menzala1980}}. It was studied again later with ordinary differential equations techniques \citelist{\cite{Tod-Moroz-1999}\cite{Moroz-Penrose-Tod-1998}}.
The existence and qualitative properties of positive ground state solutions of \((\mathcal{R}_1)\)
for an optimal range of \(p>1\)  (see theorem \ref{theoremExistenceLimitSuperlinear} above)
were studied in \cite{MVSGNLSNE} (see also \cite{MVSGGCE}).

We finally define the \emph{concentration function} \(\mathcal{C} : \R^N \to \R\) for every \(x \in \R^N\) by
\begin{equation}\label{defC}
 \mathcal{C} (x) = \mathcal{E} \bigl(V (x)\bigr),
\end{equation}
where \(\mathcal{E}(\lambda)\) is the ground state energy level of the limiting problem
\eqref{equationLimit}, as defined in \eqref{Elambda}.

\subsection{Upper bound on the energy}
We begin our study of the asymptotics of solutions by an upper bound on the mountain-pass level.

\begin{lemma}[Upper bound on the energy]
\label{lemmaUpperBound}
One has
\[
\limsup_{\varepsilon\to 0} \frac{c_\varepsilon}{\varepsilon^{N}} \le \inf_{\Lambda} \mathcal{C}.
\]
\end{lemma}

By the coerciveness property of lemma~\ref{lemmaCoerciveness}, the upper bound of lemma~\ref{lemmaUpperBound} implies immediately that if \((u_\varepsilon)_{\varepsilon > 0}\) is a family of groundstates of \eqref{equationPenalized}, then
\begin{equation}\label{lemmaUpperBound-plus}
\limsup_{\varepsilon\to 0}\frac{1}{\varepsilon^N}\int_{\R^N} \varepsilon^2 \abs{\nabla u_\varepsilon}^2 + V\abs{u_\varepsilon}^2 <\infty.
\end{equation}

\begin{proof}[Proof of lemma~\ref{lemmaUpperBound}]
Given \(a \in \Lambda\) and a nonnegative function \(v \in C^\infty_c (\R^N) \setminus \{0\}\), we define for every \(\varepsilon > 0\) the function \(v_\varepsilon \in C^\infty_c (\R^N) \setminus \{0\}\) for each \(x \in \R^N\) by
\[
  v_\varepsilon(x)=v \Big(\frac{x-a}{\varepsilon}\Big)
\]
Since \(v\) is compactly supported and the potential \(V\) is continuous at \(a\),
\begin{multline*}
 \lim_{\varepsilon \to 0} \frac{1}{\varepsilon^N} \int_{\R^N}\bigl(\varepsilon^2\abs{\nabla v_\varepsilon}^2 +V \abs{v_\varepsilon}^2\bigr)\\
 =\int_{\R^N}\bigl(\abs{\nabla v}^2+V (a) \abs{v}^2\bigr)+\lim_{\varepsilon \to 0} \int_{\R^N} \bigl(V (a + \varepsilon y) - V (a)\bigr)\abs{v (y)}^2\\
 =\int_{\R^N}\bigl(\abs{\nabla v}^2+V \abs{v}^2\bigr),
\end{multline*}
as \(\varepsilon \to 0\).

On the other hand, since the function \(v\) is nonnegative and compactly supported and the set \(\Lambda\) is open, for \(\varepsilon > 0\) small enough, \(\supp v_\varepsilon \subset \Lambda\) and
\[
  \frac{p}{2\varepsilon^{N + \alpha}}\int_{\R^N}\bigabs{I_\frac{\alpha}{2}\ast \mathscr{G}_\varepsilon(v_\varepsilon)}^2
  =\frac{p}{2\varepsilon^{N + \alpha}}\int_{\R^N} \bigabs{I_\frac{\alpha}{2}\ast (v_\varepsilon)_+^p}^2
  =\frac{p}{2} \int_{\R^N} \bigabs{I_\frac{\alpha}{2}\ast v_+^p}^2
=\frac{p}{2} \int_{\R^N} \bigabs{I_\frac{\alpha}{2}\ast \abs{v}^p}^2.
\]
We define the path \(\gamma_\varepsilon \in C([0, \infty), H^1_{V} (\R^N))\) for \(\tau \ge 0\) by  \(\gamma_\varepsilon (\tau)=\tau v_\varepsilon\). It is clear that \(\gamma_\varepsilon \in\Gamma_\varepsilon\) after appropriate reparametrisation, and
\[
 \limsup_{\varepsilon \to 0} \frac{c_\varepsilon}{\varepsilon^N} \le \sup_{t \in [0, 1]} \frac{\J_\varepsilon \bigl(\gamma (t)\bigr)}{\varepsilon^N} \le \max_{\tau \ge 0} \I_{V(a)}(\tau v) + o (1)
\]
(the set of paths \(\Gamma_\varepsilon\) and \(c_\varepsilon\) were defined in the statement of proposition~\ref{propositionPenalizedExistence}).
By taking the infimum over \(v \in C^\infty_c (\R^N)\), we obtain in view of the characterization \eqref{ElambdaCinfty} of the limiting energy \(\mathcal{E} (V (a))\) and of the definition of the concentration function \(\mathcal{C}\),
\[
 c_\varepsilon \le \mathcal{E} \bigl(V (a)\bigr) = \mathcal{C} (a),
\]
and we conclude by taking the infimum over \(a \in \Lambda\).
\end{proof}

\subsection{Lower bound on the energy}
In order to understand the behaviour of families of solutions we will rely on a  some lower bound on the energy of a sequence of solutions concentrating along a sequence of points.

\begin{proposition}
\label{propositionLowerBound}
Let \((\varepsilon_n)_{n \in \N}\) be a sequence of positive numbers converging to \(0\), let  \((u_n)_{n \in \N}\) be a sequence of nonnegative solutions in \(H^1_{V} (\R^N)\) of \((\mathcal{Q}_{\varepsilon_n})\),  and for \(j \in \{1, \dotsc, k\}\), let \((a^j_n)_{n \in \N}\) be a sequence in \(\R^N\) that converges to \(a^j_* \in \R^N\). If for every \(j \in \{1, \dotsc, k\}\), \(V (a^j_*) > 0\), for every  \(\ell \in \{1, \dotsc, k\} \setminus \{j\}\),
\[
  \lim_{n \to \infty} \frac{\abs{a^j_n - a^\ell_n}}{\varepsilon_n} = \infty,
\]
and for some \(\rho > 0\),
\[
\liminf_{n \to \infty} \norm{u_n}_{L^\infty (B_{\varepsilon_n \rho} (a^j_n))} + \varepsilon_n^{-\alpha} \norm{ I_\alpha \ast \mathscr{G}_{\varepsilon_n} (u_n)}_{L^\infty (B_{\varepsilon_n \rho} (a^j_n))} > 0,
\]
then \(a^j_* \in \Bar{\Lambda}\) and
\[
  \liminf_{n \to \infty} \frac{\J_\varepsilon (u_n)}{\varepsilon_n^N}
\ge \sum_{j = 1}^k \mathcal{C} (a^j_*).
\]
\end{proposition}

In the proof of proposition~\ref{propositionLowerBound}, we shall rely on a the \(C^\infty\)--convergence of Riesz potentials far from the support of the density.

\begin{proposition}
\label{propositionInteriorExteriorRieszconvergence}
Assume that \((f_n)_{n \in \N}\) is a sequence in \(L^q (\R^N ; \abs{x}^\beta \dif  x)\) and that
\(f_n \weakto f\) in \(L^q (\R^N; \abs{x}^\beta \dif  x)\) as \(n \to \infty\).
If for every \(n \in \N\), \(f_n = 0\) on \(B_R\) and if \(\alpha q < N + \beta\), then for every \(k \in \N\),
\[
  D^k (I_\alpha \ast f_n) \to D^k (I_\alpha \ast f)
\]
uniformly on every compact subset of \(B_R\).
\end{proposition}
\begin{proof}
Let \(K \subset B_R\) be compact.
For every \(k \in \N\) and \(x \in K\), we estimate by the H\"older inequality
\[
\begin{split}
 \abs{D^k (I_\alpha \ast f_n) (x)} & \le C \int_{\R^N \setminus B_{R}} \frac{1}{\abs{x - y}^{N - \alpha + k}} \abs{f_n (y)} \dif y\\
& \le C' \int_{\R^N \setminus B_{R}} \frac{1}{\abs{y}^{N - \alpha + k}} \abs{f_n (y)} \dif y\\
& \le C' \Bigl(\int_{\R^N \setminus B_{R}} \frac{1}{\abs{y}^{\frac{q (N - \alpha + k) + \beta}{q - 1}}}\dif y \Bigr)^{1 - \frac{1}{q}}  \Bigl(\int_{\R^N} \abs{f_n (y)}^q \abs{y}^{\beta} \dif y\Bigr)^{\frac{1}{q}}.
\end{split}
\]
Since \(N + \beta > \alpha q\)
\[
 \int_{\R^N \setminus B_{R}} \frac{1}{\abs{y}^{\frac{q (N - \alpha + k) + \beta}{q - 1}}} \dif y < \infty,
\]
and the conclusion follows from a standard application of Ascoli's compactness criterion.
\end{proof}

We will also use a Liouville theorem for problems penalized on a half-space.

\begin{lemma}
\label{lemmaLiouvilleHalf}
Let \(H \subset \R^N\) be a half-space.
Assume that \(\frac{N+\alpha}{N}< p<\frac{N+\alpha}{(N-2)_+}\).
If \(v \in H^1 (\R^N)\) satisfies the equation
\[
 - \Delta v + \lambda v = \bigl(I_\alpha \ast (\chi_H v_+^p)\bigr) \chi_H v_+^{p - 1} \quad\text{in \(\R^N\)},
\]
then \(v = 0\).
\end{lemma}
\begin{proof}
By a linear isometric change of variable, we can assume that \(H = \R^{N - 1} \times (0, \infty)\).
In view of a regularity argument similar to \cite{MVSGNLSNE}*{proposition~4.1}, \(v \in H^2 (\R^N)\) and
in particular, \(\partial_N v\) is an admissible test function in the equation. We obtain then by testing the equation against \(\partial_N v\)
\[
\begin{split}
0& = \frac{1}{p} \int_{\R^{N - 1} \times (0, \infty)} \int_{\R^{N - 1} \times (0, \infty)}
I_\alpha (x - y) \bigl(v (y)\bigr)_+^p \bigl(\partial_N (v)_+^p\bigr) (x) \dif x\dif y\\
& = \frac{1}{p} \int_{\R^{N - 1}} \int_{\R^{N - 1} \times (0, \infty)}
 I_\alpha \bigl((z, 0) - y\bigr) \bigl(v (z)\bigr)_+^p \bigl(v(y)\bigr)_+^p\dif z\dif y.
\end{split}
\]
We deduce therefrom that either \(v = 0\) on \(\R^{N - 1} \times \{0\}\) or \(v = 0\) on \(\R^{N - 1} \times (0, \infty)\).
We conclude by the strong maximum principle.
\end{proof}

\begin{proof}[Proof of proposition~\ref{propositionLowerBound}]
We can assume without loss of generality that
\[
  \liminf_{n \to \infty} \frac{\J_\varepsilon (u_n)}{\varepsilon_n^N}  = \limsup_{n \to \infty} \frac{\J_\varepsilon (u_n) }{\varepsilon_n^N} < \infty.
\]
By the coerciveness of the functional \(J_{\varepsilon_n}\) on critical points (lemma~\ref{lemmaCoerciveness}), this implies that
\[
  \limsup_{n \to \infty} \frac{1}{\varepsilon_n^N}\int_{\R^N} \varepsilon_n^2 \abs{\nabla u_n}^2 + V \abs{u_n}^2 < \infty.
\]

We define for every \(n \in \N\) and \(j \in \{1, \dotsc, k\}\) the rescaled solution \(v^j_n \in H^1_\mathrm{loc} (\R^N)\) for every \(y \in \R^N\) by
\[
  v^j_n(y)=u_n(a^j_n+\varepsilon_n y).
\]
Since \(u_n \) solves the penalized problem \((\mathcal{Q}_{\varepsilon_n})\), the function \(v^j_n\) satisfies weakly the rescaled equation
\[
  - \Delta v^j_n + V^j_n v^j_n =  \bigl(I_\alpha \ast \mathscr{G}^j_n (v^j_n)\bigr) \mathscr{g}^j_n (v^j_n) \quad\text{in \(\R^N\)},
\]
where for \(y \in \R^N\) and \(s \in \R\), the rescaled potential and nonlinearities are defined by \(V_n (y) = V (a_n + \varepsilon_n y)\),
\begin{align*}
 g^j_n (y, s) &= g_{\varepsilon_n} (a^j_n + \varepsilon_n y, s),&
 G^j_n (y, s) &= G_{\varepsilon_n} (a^j_n + \varepsilon_n y, s),
\end{align*}
and \(\mathscr{g}^j_n\) and \(\mathscr{G}^j_n\) are the corresponding nonlinear superposition operators.

For every \(R > 0\), a change of variable shows that
\[
 \int_{B_R} \abs{\nabla v^j_n}^2 + V^j_n \abs{v^j_n}^2
= \frac{1}{\varepsilon_n^N} \int_{B_{\varepsilon_n R} (a^j_n)} \varepsilon_n^2 \abs{\nabla u_n}^2 + V \abs{u_n}^2.
\]
Since \(V\) is positive and continuous at \(a^j_* = \lim_{n \to \infty} a^j_n\), we have
\[
 \lim_{n \to \infty} \frac{\displaystyle \int_{B_R} V^j_n \abs{v^j_n}^2}{\displaystyle \int_{B_R} V (a^j_*) \abs{v^j_n}^2} = 1,
\]
and therefore
\[
 \liminf_{n \to \infty} \int_{B_R} \abs{\nabla v^j_n}^2 + V (a^j_*) \abs{v^j_n}^2
\le \liminf_{n \to \infty} \frac{1}{\varepsilon_n^N} \int_{\R^N} \varepsilon_n^2 \abs{\nabla u_{n}}^2 + V \abs{u_{n}}^2.
\]
By taking a subsequence if necessary and by a diagonal argument, there exists \(v^j_* \in H^1_\mathrm{loc} (\R^N)\) such that for every \(R > 0\), \(v^j_n \weakto v^j_*\) weakly in \(H^1 (B_R)\) as \(n \to \infty\) and
\[
  \int_{B_R} \abs{\nabla v^j_*}^2 + V (a_*) \abs{v^j_*}^2
\le \liminf_{n \to \infty} \frac{1}{\varepsilon_n^N} \int_{\R^N} \varepsilon_n^2 \abs{\nabla u_{n}}^2 + V \abs{u_{n}}^2 < \infty.
\]
This implies in particular that \(v^j_* \in H^1 (\R^N)\).
Moreover, by the Rellich--Kondrachov compactness theorem (see for example \citelist{\cite{LiebLoss2001}*{theorem 8.9}\cite{Adams1975}*{theorem 6.2}\cite{Willem2013}*{theorem 6.4.6}}), for every \(q \ge 1\) such that \(\frac{1}{q} > \frac{1}{2} - \frac{1}{N}\) and for every \(R > 0\), we have \(v^j_n \to v^j_*\) in \(L^q  (B_R)\) as \(n \to \infty\).

Defining for \(n \in \N\) and \(j \in \{1, \dotsc, k\}\) the rescaled sets
\[
 \Lambda^j_n = \{y \in \R^N : a^j_n + \varepsilon_n y \in \Lambda \},
\]
and taking into account the smoothness of the boundary of \(\Lambda\), we can assume that \(\chi_{\Lambda^j_n} \to \chi_{\Lambda^j_*}\) almost everywhere as \(n \to \infty\), where \(\Lambda^j_\ast\) is either \(\R^N\), a half-space or \(\emptyset\).

Observe that for every \(x \in \R^N\) and \(s \in \R\), by \eqref{eqHtozero}, \(\lim_{n \to \infty} g_n (x, s) = \chi_{\Lambda_*} (x) s_+^{p - 1}\) and for every \(n \in \N\), it holds by \eqref{assumptGBoundp} that \(\abs{\mathscr{g}^j_n (v^j_n)} \le (v^j_n)_+^{p - 1}\) and thus by Lebesgue's dominated convergence theorem,
\begin{align}
\label{eqRescaledLocal}
  \mathscr{g}^j_n (v^j_n) & \to \chi_{\Lambda^j_*} (v^j_*)_+^{p - 1} & &\text{in \(L^q (B_{2 R})\) with \(\tfrac{1}{q} > (p - 1) (\tfrac{1}{2} - \tfrac{1}{N})\)}
\end{align}
and
\begin{align*}
  p \mathscr{G}^j_n (v^j_n) & \to \chi_{\Lambda^j_*}  (v^j_*)_+^p & &\text{in \(L^q (B_{2 R})\) with \(\tfrac{1}{q} > p (\tfrac{1}{2} - \tfrac{1}{N})\)}.
\end{align*}
By the Hardy--Littlewood--Sobolev inequality (proposition~\ref{propositionHLS}), this implies that
\begin{align}
\label{eqRescaledRieszBall}
  p I_\alpha \ast \bigl(\chi_{B_{2 R}} \mathscr{G}^j_n (v^j_n)\bigr) &\to I_\alpha \ast \bigl( \chi_{B_{2 R} \cap \Lambda^j_*} (v^j_*)_+^p\bigr)&
  &\text{in \(L^q (B_R)\) with \(\tfrac{1}{q} > p (\tfrac{1}{2} - \tfrac{1}{N}) - \tfrac{\alpha}{N}\)}.
\end{align}

By definition of the rescaled penalized nonlinearity \(G_n^j\) and by the rescaled Sobolev inequality (proposition~\ref{propositionSobolevScaling}),
since \(\frac{N - 2}{N + \alpha} \le \frac{1}{p} \le \frac{N}{N + \alpha}\), we have
\[
 \int_{\Lambda^j_n} \abs{\mathscr{G}^j_n (v^j_n)}^\frac{2 N}{N + \alpha}
 \le \frac{1}{\varepsilon_n^N} \int_{\Lambda} \abs{u_n}^\frac{2 N p}{N + \alpha} \dif x
 \le \Bigl( \frac{1}{\varepsilon_n^N} \int_{\R^N} \varepsilon^2 \abs{\nabla u_n}^2 + V \abs{u_n}^2 \Bigr)^\frac{N p}{N + \alpha}.
\]
Since \(\mathscr{G}^j_n(v^j_n) \to \frac{1}{p} \chi_{\Lambda^j_*} (v^j_*)_+^p\) in \(L^\frac{2 N}{N + \alpha}_{\mathrm{loc}} (\R^N)\), as \(n \to \infty\),
\[
\chi_{\Lambda^j_n} \mathscr{G}^j_n(v^j_n) \weakto \frac{1}{p} \chi_{\Lambda^j_*} (v^j_*)_+^p\qquad\text{ in \(L^{\frac{2 N}{N + \alpha}} (\R^N)\)}.
\]
By proposition~\ref{propositionInteriorExteriorRieszconvergence} with \(q = \frac{2 N}{N + \alpha}\) and \(\beta = 0\), we have, as \(n \to \infty\),
\begin{align}
\label{eqRescaledRieszLambda}
  p I_\alpha \ast \bigl(\chi_{\Lambda^j_n \setminus B_{2 R}} \mathscr{G}^j_n (v^j_n)\bigr)
& \to  I_\alpha \ast \bigl(\chi_{\Lambda^j_* \setminus B_{2 R}}  (v^j_*)_+^p\bigr)
& & \text{in \(L^\infty (B_R)\)}.
\end{align}
Finally, by \((H_2)\),
\[
\begin{split}
 \int_{\R^N \setminus \Lambda^j_n  } \abs{G^j_n (y, v^j_n (y))}^2 \abs{y}^{\alpha} \dif x
& = \frac{1}{\varepsilon_n^{N - \alpha}} \int_{\R^N \setminus \Lambda} \abs{G_{\varepsilon_n} (x, u_n (x))}^2 \abs{x - a^j_n}^{\alpha} \dif x\\
& \le \frac{1}{\varepsilon_n^{N - \alpha}} \Bigl(\sup_{x \in \R^N \setminus \Lambda} \frac{\abs{x - a^j_n}}{\abs{x}} \Bigr)^\alpha\int_{\R^N \setminus \Lambda} \abs{H_{\varepsilon_n} (x) u_n (x)}^2 \abs{x}^{\alpha} \dif x\\
& \le \frac{\kappa}{\varepsilon_n^{N}} \Bigl(\sup_{x \in \R^N \setminus \Lambda} \frac{\abs{x - a^j_n}}{\abs{x}} \Bigr)^\alpha \int_{\R^N \setminus \Lambda} \bigl(\varepsilon_n^2 \abs{\nabla u_n}^2 + V \abs{u_n}^2\bigr).
\end{split}
\]
Since \(\mathscr{G}^j_n (v^j_n) \chi_{\R^N \setminus \Lambda^j_n} \to 0\) in \(L^2_{\mathrm{loc}} (\R^N)\), we have by proposition~\ref{propositionInteriorExteriorRieszconvergence}, as \(n \to \infty\),
\begin{align}
\label{eqRescaledRieszOutside}
  p I_\alpha \ast \bigl(\chi_{\R^N \setminus (\Lambda^j_n \cup B_{2R})} \mathscr{G}^j_n (v^j_n) \bigr)
&\to 0 & & \text{in \(L^\infty (B_R)\)}.
\end{align}
Summarizing \eqref{eqRescaledRieszBall}, \eqref{eqRescaledRieszLambda} and \eqref{eqRescaledRieszOutside}, we obtain, as \(n \to \infty\),
\begin{align*}
  p I_\alpha \ast \mathscr{G}^j_n (v^j_n) \to I_\alpha \ast \bigl(\chi_{\Lambda^j_*}  (v^j_*)_+^p\bigr) & & \text{in \(L^q (B_R)\) with \(\tfrac{1}{q} > p (\tfrac{1}{2} - \tfrac{1}{N}) - \tfrac{\alpha}{N}\)}.
\end{align*}
In view of \eqref{eqRescaledLocal}, we have, as \(n \to \infty\),
\begin{multline}
\label{eqRescaledTotal}
  p \bigl(I_\alpha \ast \mathscr{G}^j_n (v^j_n)\bigr) \mathscr{g}^j_n (v^j_n)
  \to \bigl(I_\alpha \ast \chi_{\Lambda^j_*} (v^j_*)_+^p\bigl) \chi_{\Lambda^j_*} \abs{v^j_*}^{p - 2} v^j_*\\
 \qquad \qquad\text{in \(L^q (B_{R})\) with \(\textstyle \frac{1}{q} > \bigl(p(\frac{1}{2} - \frac{1}{N}) - \frac{\alpha}{N}\bigr)_+ + \bigl((p - 1) (\frac{1}{2} - \frac{1}{N})\bigr)_+\)}.
\end{multline}
Since \(\frac{1}{p} \ge \frac{N - 2}{N + \alpha}\), we can take \(q = 1\) and \(v^j_*\) satisfies
\[
 - \Delta v^j_*  + V (a^j_*) v^j_* = \bigl(I_\alpha \ast \chi_{\Lambda^j_*} (v^j_*)_+^p\bigl) \chi_{\Lambda^j_*} (v^j_*)_+^{p - 1}.
\]
By the adaptation of the classical bootstrap argument adapted to the Choquard equation (see \cite{MVSGNLSNE}*{proposition~4.1}),
\(v^j_n \to v^j_*\) and \(I_\alpha \ast (\mathscr{G}^j_n (v^j_n)) \to I_\alpha \ast \chi_{\Lambda^j_*} (v^j_*)_+^p\) as \(n \to \infty\) uniformly on compact sets of \(\R^N\). Hence,
\begin{multline*}
  \norm{v^j_*}_{L^\infty (B_\rho)} + \norm{ I_\alpha \ast (\chi_{\Lambda^j_*} (v^j_*)_+^p)}_{L^\infty (B_\rho))}\\
  \hspace{-10ex}=  \lim_{n \to \infty} \norm{v^j_n}_{L^\infty (B_\rho)} + \norm{I_\alpha \ast (\mathscr{G}^j_n (v^j_n))}_{L^\infty (B_\rho))}\\
   = \lim_{n \to \infty} \norm{u_n}_{L^\infty (B_{\varepsilon_n \rho} (a^j_n))} + \varepsilon_n^{-\alpha} \norm{ I_\alpha \ast \mathscr{G}_{\varepsilon_n} (u_n)}_{L^\infty (B_{\varepsilon_n \rho} (a^j_n))} > 0.
\end{multline*}

Therefore, \(v^j_* \ne 0\) and the set \(\Lambda^j_*\) cannot be empty.
By lemma~\ref{lemmaLiouvilleHalf}, it can neither be a half-space. Hence, we conclude that \(\Lambda^j _* = \R^N\) and \(a^j_* \in \Bar{\Lambda}\), that the function \(v^j_*\) satisfies the equation
\begin{align*}
  - \Delta v^j_*  + V (a^j_*) v^j_* &= \bigl(I_\alpha \ast (v^j_*)_+^p\bigl) (v^j_*)_+^{p - 1}&
  &\text{in \(\R^N\)},
\end{align*}
and that, since \(v^j_* \ge 0\),
\begin{multline}
\label{equationAsymptoticsInside}
  \liminf_{n \to \infty} \frac{1}{2 \varepsilon_n^N} \int_{B_{\varepsilon_n R} (a^j_n)} \Bigl(\varepsilon_n^2\abs{\nabla u_n}^2 + V \abs{u_n}^2 - \frac{p}{\varepsilon_n^\alpha} \bigl(I_\alpha \ast \mathscr{G}_{\varepsilon_n} (u_n)\bigr) \mathscr{G}_{\varepsilon_n} (u_n)\Bigr) \\
  \ge \frac{1}{2} \int_{B_{R}} \abs{\nabla v^j_*}^2 + V (a^j_*) \abs{v^j_*}^2 - \frac{1}{2 p} \int_{B_R} \bigl( I_\alpha \ast (v^j_*)_+^p \bigr) (v^j_*)_+^p\\
  = \frac{1}{2} \int_{B_{R}} \abs{\nabla v^j_*}^2 + V (a^j_*) \abs{v^j_*}^2 - \frac{1}{2 p} \int_{B_R} \bigl( I_\alpha \ast \abs{v^j_*}^p \bigr) \abs{v^j_*}^p\\
  \ge \mathcal{C} (a^j_*) - \frac{1}{2}\int_{\R^N \setminus B_R} \abs{\nabla v^j_*}^2 + V (a^j_*) \abs{v^j_*}^2.
\end{multline}

In order to study the integral outside \(B_{\varepsilon_n R} (a^j_n)\), we choose \(\eta \in C^\infty (\R^N)\) such that \(0 \le \eta \le 1\), \(\eta = 0\) on \(B_1\) and \(\eta = 1\) on \(\R^N \setminus B_2\), and we define for \(n \in \N\) and \(R > 0\) the function \(\psi_{n, R} \in C^\infty (\R^N)\) for \(x \in \R^N\) by
\[
 \psi_{n, R} (x) = \prod_{j=1}^k \eta \Bigl( \frac{x - a^j_n}{\varepsilon_n R} \Bigr).
\]
Since \(u_n\) is a solution to the penalized problem \((\mathcal{Q}_{\varepsilon_n})\), we have by taking \(\psi_{n, R} u_n\) as a test function in the penalized problem \((\mathcal{Q}_{\varepsilon_n})\),
\begin{multline*}
  \int_{\R^N \setminus \cup_{j=1}^{k} B_{\varepsilon_n R} (a^j_n)} \varepsilon_n^2 \psi_{n, R} \abs{\nabla u_n}^2 + V \psi_{n, R} \abs{u_n}^2\\
  = \frac{p}{\varepsilon_n^{\alpha}} \int_{\R^N \setminus \cup_{j=1}^{k} B_{\varepsilon_n R} (a^j_n)} \bigl(I_\alpha \ast  \mathscr{G}_{\varepsilon_n} (u_n) \bigr) \mathscr{g}_{\varepsilon_n} (u_n) u_n \psi_{n, R} - \int_{\R^N \setminus \cup_{j=1}^{k} B_{\varepsilon_n R} (a^j_n)} \varepsilon_n^2 u_n \scalprod{\nabla \psi_{n, R}}{\nabla u_n}.
\end{multline*}
Hence, in view of the superlinearity assumption \((g_3)\) on the penalized nonlinearity
\begin{multline*}
\int_{\R^N \setminus \cup_{j=1}^{k} B_{\varepsilon_n R} (a^j_n)} \frac{1}{2} \bigl( \varepsilon_n^2 \abs{\nabla u_n}^2 + V \abs{u_n}^2\bigr) - \frac{p}{2 \varepsilon_n^\alpha} \bigl(I_\alpha \ast \mathscr{G}_{\varepsilon_n} (u_n)\bigr) \mathscr{G}_{\varepsilon_n} (u_n)\\
\ge\frac{1}{2}
 \int_{\R^N \setminus \cup_{j=1}^{k} B_{\varepsilon_n R} (a^j_n)}  \varepsilon_n^2 \psi_{n, R} \abs{\nabla u_n}^2 + V \psi_{n, R} \abs{u_n}^2 -  \frac{p}{\varepsilon_n^\alpha} \bigl(I_\alpha \ast \mathscr{G}_{\varepsilon_n} (u_n)\bigr) \mathscr{g}_{\varepsilon_n} (u_n)u_n\\
= - \frac{1}{2}\int_{\R^N \setminus \cup_{j=1}^{k} B_{\varepsilon_n R} (a^j_n)} \varepsilon_n^2
\scalprod{u_n \nabla \psi_{n, R}}{\nabla u_n} + \frac{p}{\varepsilon_n^\alpha}\bigl(I_\alpha \ast \mathscr{G}_{\varepsilon_n} (u_n)\bigr) \mathscr{g}_{\varepsilon_n} (u_n)u_n (1 - \psi_{n, R}).
\end{multline*}
By scaling, if \(n\) is large enough so that \(B_{\varepsilon_n R} (a^j_n) \cap B_{\varepsilon_n R} (a^\ell_n) = \emptyset\),
\begin{multline*}
  \frac{1}{\varepsilon_n^N} \int_{\R^N \setminus \cup_{j=1}^{k} B_{\varepsilon_n R} (a^j_n)} \frac{1}{2} \bigl( \varepsilon_n^2 \abs{\nabla u_n}^2 + V \abs{u_n}^2\bigr) - \frac{p}{2 \varepsilon_n^\alpha} \bigl(I_\alpha \ast \mathscr{G}_{\varepsilon_n} (u_n)\bigr) \mathscr{G}_{\varepsilon_n} (u_n)\\
  \ge - \frac{1}{2}\sum_{j=1}^k \int_{B_{2R} \setminus B_R} v^j_n \scalprod{\nabla \eta_R}{\nabla v^j_n} + \bigl(I_\alpha \ast \mathscr{G}^j_{n} (v^j_n)\bigr) \mathscr{g}^j_{n} (v^j_n)v^j_n,
\end{multline*}
where \(\eta_R \in C^\infty (\R^N)\) is defined for \(R > 0\) and \(y \in \R^N\) by \(\eta_R (y) = \eta (y/R)\).
Therefore, since \(v^j_n \weakto v^j_*\) in \(H^1 (B_R)\) as \(n \to \infty\), \(v^j_n \to v^j_*\) in \(L^2 (B_R)\) as \(n \to \infty\),
and by the convergence \eqref{eqRescaledTotal}
\[
\bigl(I_\alpha \ast \mathscr{G}^j_n (v^j_n)\bigr) \mathscr{g}^j_n (v^j_n) v^j_n
  \to \bigl(I_\alpha \ast (v^j_*)_+^p\bigl) (v^j_*)_+^p
 \qquad \qquad\text{in \(L^1 (B_{R})\)},
\]
and therefore
\begin{multline}
\label{equationAsymptoticsOutside}
  \liminf_{n \to \infty} \frac{1}{\varepsilon_n^N} \int_{\R^N \setminus \cup_{j=1}^{k} B_{\varepsilon_n R} (a^j_n)} \frac{1}{2} \bigl( \varepsilon_n^2 \abs{\nabla u_n}^2 + V \abs{u_n}^2\bigr) - \frac{p}{2 \varepsilon_n^\alpha} \bigl(I_\alpha \ast \mathscr{G}_{\varepsilon_n} (u_n)\bigr) \mathscr{G}_{\varepsilon_n} (u_n)\\
  \ge -\frac{1}{2} \sum_{j=1}^k \int_{B_{2 R} \setminus B_R} \scalprod{v^j_* (y) \nabla \eta_R}{\nabla v^j_*} + \bigl(I_\alpha \ast (v^j_*)_+^p \bigr) (v^j_*)_+^p
\end{multline}

In order to conclude, we combine \eqref{equationAsymptoticsInside} and \eqref{equationAsymptoticsOutside}, to obtain
\begin{multline*}
  \liminf_{n \to \infty} \frac{\J_\varepsilon (u_n)}{\varepsilon_n^N}
  \ge \sum_{j=1}^k \mathcal{C} (a^j_*) - \frac{1}{2}\sum_{j=1}^k \biggl( \int_{B_{2 R} \setminus B_R} \scalprod{v^j_*\nabla \eta_R }{\nabla v^j_*} + \bigl(I_\alpha \ast (v^j_*)_+^p \bigr) (v^j_*)_+^p\\ + \int_{\R^N \setminus B_R} \abs{\nabla v^j_*}^2 + V (a^j_*) \abs{v^j_*}^2\biggr)
\end{multline*}
and we observe that the right-hand side goes to \(0\) as \(R \to \infty\) by Lebesgue's dominated convergence theorem since \(v^j_* \in H^1 (\R^N)\) and \((I_\alpha \ast (v^j_*)_+^p) (v^j_*)_+^p\) is integrable by the Hardy--Littlewood--Sobolev inequality (proposition~\ref{propositionHLS}).
\end{proof}

\subsection{Asymptotic behavior of solutions}

Here and in the sequel, we choose a bounded open set \(U \subset \R^N\) such that \(\inf_{U} V > 0\) and \(\Bar{\Lambda} \subset U\). The following statement summarizes concentration estimates we have obtained so far.

\begin{proposition}\label{propositionConcentration}
Let \(\rho > 0\).
There exists a family of points \((a_\varepsilon)_{\varepsilon > 0}\) in \(\Lambda\) such that
\begin{gather*}
  \liminf_{\varepsilon \to 0} \norm{u_\varepsilon}_{L^\infty (B_{\varepsilon \rho} (a_\varepsilon))} > 0,\\
   \lim_{\varepsilon \to 0} V (a_\varepsilon) = \inf_{\Lambda} V,\\
  \lim_{\varepsilon \to 0} \dist (a_\varepsilon, \R^N \setminus \Lambda) > 0,\\
 \lim_{\substack{R \to \infty\\ \varepsilon \to 0}} \norm{u_\varepsilon}_{L^\infty (U \setminus B_{\varepsilon R} (a_\varepsilon))} + \frac{1}{\varepsilon^\alpha} \norm{I_\alpha \ast \mathscr{G}_\varepsilon (u_\varepsilon)}_{L^\infty (U \setminus B_{\varepsilon R} (a_\varepsilon))}= 0.
\end{gather*}
\end{proposition}

\begin{proof}
We begin by showing that the solutions do not converge uniformly to \(0\) in \(\Lambda\).
By testing the equation against \(u_\varepsilon\), and applying the growth assumption \((g_2)\) and the Young inequality, we estimate for every \(\lambda > 0\)
\begin{multline*}
\int_{\R^N} \bigl(\varepsilon^2\abs{\nabla u_\varepsilon}^2 + V \abs{u_\varepsilon}^2\bigr) =
\frac{p}{\varepsilon^\alpha} \int_{\R^N} \bigl(I_\alpha \ast \mathscr{G}_\varepsilon(u_\varepsilon)\bigr)
\mathscr{g}_\varepsilon (u_\varepsilon) u_\varepsilon\\
\le
\Bigl(1 + \lambda\Bigr)\frac{p}{\varepsilon^\alpha} \int_{\R^N} \bigabs{I_\frac{\alpha}{2} \ast (H_\varepsilon (u_\varepsilon)_+)}^2 + \Bigl(1 + \frac{1}{\lambda}\Bigr)
\frac{1}{\varepsilon^\alpha} \int_{\R^N}\bigabs{I_\frac{\alpha}{2}\ast (\chi_{\Lambda} (u_\varepsilon)_+^p)}^2.
\end{multline*}
In view of the weighted Hardy--Littlewood--Sobolev inequality (proposition~\ref{propositionHLSW}) and of our hypothesis \((H_2)\), we have
\begin{equation}
\label{eqLinftyLambdac}
 \frac{p}{\varepsilon^\alpha} \int_{\R^N} \bigabs{I_\frac{\alpha}{2} \ast (H_\varepsilon (u_\varepsilon)_+)}^2
 \le \frac{C_\alpha p}{\varepsilon^\alpha} \int_{\R^N}  \abs{u_\varepsilon}^2 H_\varepsilon (x)^2 \abs{x}^\alpha \dif x
 \le C_\alpha p \kappa \int_{\R^N} \big(\varepsilon^2\abs{\nabla u_\varepsilon}^2 + V \abs{u_\varepsilon}^2\big)
\end{equation}
On the other hand, by the Hardy--Littlewood--Sobolev inequality of proposition~\ref{propositionHLS} and by the rescaled Sobolev inequality of proposition~\ref{propositionSobolevScaling},
\begin{equation}
\label{eqLinftyLambda}
\begin{split}
 \frac{1}{\varepsilon^\alpha} \int_{\R^N}\bigabs{I_\frac{\alpha}{2}\ast (\chi_{\Lambda} (u_\varepsilon)_+^{p})}^2
 &\le \frac{C}{\varepsilon^\alpha} \Big(\int_{\Lambda} \abs{u_\varepsilon}^\frac{2N p}{N+\alpha}\Big)^{1+\frac{\alpha}{N}}\\
 &\le \frac{C}{\varepsilon^\alpha} \norm{u_\varepsilon}_{L^\infty (\Lambda)}^{2 p - 2} \Big(\int_{\Lambda} \abs{u_\varepsilon}^\frac{2N}{N+\alpha}\Big)^{1+\frac{\alpha}{N}}\\
 &\le C' \norm{u_\varepsilon}_{L^\infty (\Lambda)}^{2 p - 2} \int_{\R^N} \varepsilon^2 \abs{\nabla u_\varepsilon}^2 + V \abs{u_\varepsilon}^2.
 \end{split}
\end{equation}
Since \(u_\varepsilon \ne 0\), we deduce from \eqref{eqLinftyLambdac} and \eqref{eqLinftyLambda} that
\[
 \bigl(1 - C_\alpha p\kappa  (1 + \lambda)\bigr) \le C' \Bigl(1 + \frac{1}{\lambda}\Bigr) \norm{u_\varepsilon}_{L^\infty}^{2p - 2}.
\]
Since \(C_\alpha p\kappa < 1\) by \((H_2)\), this gives the first conclusion if \(\lambda > 0\) is taken small enough.
There exists thus a family of points \((a_\varepsilon)_{\varepsilon > 0}\) in \(\Lambda\) such that
\[
  \liminf_{\varepsilon \to 0} \norm{u_\varepsilon}_{L^\infty (B_{\varepsilon \rho} (a_\varepsilon))} > 0.
\]

Next, assume that the sequence of points \((a_{\varepsilon_n})_{n \in \N}\) in \(U\) converges to \(a_* \in \Bar{\Lambda}\) and \(\lim_{n \to \infty} \varepsilon_n = 0\).
By the lower bound on the energy of proposition~\ref{propositionLowerBound},
\[
  \liminf_{n \to \infty} \frac{1}{\varepsilon_n^N} \J_{\varepsilon_n} (u_{\varepsilon_n}) \ge \mathcal{C} (a_*) = \lim_{n \to \infty} \mathcal{C} (a_n).
\]
In view of the upper bound of the energy of lemma~\ref{lemmaUpperBound}, this implies that
\[
  \lim_{n \to \infty} \mathcal{C} (a_n) = \inf_{\Lambda} \mathcal{C}.
\]
Since the set \(\Bar{U}\) is compact, this proves that \(\lim_{\varepsilon \to 0} \mathcal{C} (a_\varepsilon) = \inf_{\Lambda} \mathcal{C}\). By the definition in \eqref{defC} of the concentration function \(\mathcal{C}\) and by proposition~\ref{propositionScaling}, this implies that \(\lim_{\varepsilon \to 0} V (a_\varepsilon) = \inf_{\Lambda} V\), and thus, as \(\inf_{\Lambda} V < \inf_{\partial \Lambda} V\), that \(\lim_{\varepsilon \to 0} \dist (a_\varepsilon, \R^N \setminus \Lambda) > 0\).

Finally, we assume by contradiction that
\[
\limsup_{\substack{R \to \infty\\ \varepsilon \to 0}} \norm{u_\varepsilon}_{L^\infty (U \setminus B_{\varepsilon R} (a_\varepsilon))} + \varepsilon^{-\alpha} \norm{I_\alpha \ast \mathscr{G}_\varepsilon (u_\varepsilon)}_{L^\infty (U \setminus B_{\varepsilon R} (a_\varepsilon))} > 0,
\]
Then, there exists a sequence of positive numbers \((\varepsilon_n)_{n \in \N}\) converging to \(0\) and a sequence of points \((x_n)_{n \in \N}\) in \(U\) such that
\[
  \liminf_{\varepsilon \to 0} \norm{u_{\varepsilon_n}}_{L^\infty (B_{\varepsilon_n \rho} (x_n))} + \varepsilon_n^{-\alpha} \norm{I_\alpha \ast \mathscr{G}_{\varepsilon_n} (u_{\varepsilon_n})}_{L^\infty (B_{\varepsilon_n \rho} (x_n))} > 0
\]
and
\[
  \lim_{n \to \infty} \frac{\abs{x_n - a_{\varepsilon_n}}}{\varepsilon_n} = \infty.
\]
Since \(\Bar{U}\) is compact, we can assume that the sequence \((x_n)_{n \in \N}\) converges to \(x_*\) and that \((a_{\varepsilon_n})_{n \in \N}\) converges to \(a_*\).
By the lower bound on the energy of proposition~\ref{propositionLowerBound}, \(x_* \in \Bar{\Lambda}\), \(a_* \in \Bar{\Lambda}\), and
\[
  \liminf_{n \to \infty} \frac{1}{\varepsilon_n^N} \J_{\varepsilon_n} (u_{\varepsilon_n}) \ge \mathcal{C} (a_*) + \mathcal{C} (x_*) \ge 2 \inf_{\Lambda} \mathcal{C}.
\]
On the other hand, the upper bound lemma~\ref{lemmaUpperBound} shows that
\[
  \liminf_{n \to \infty} \frac{1}{\varepsilon_n^N} \J_{\varepsilon_n} (u_{\varepsilon_n}) \le \inf_{\Lambda} \mathcal{C}.
\]
This brings a contradiction since \(\inf_{\Lambda} \mathcal{C} > 0\) in view of the definition of \(\mathcal{C}\), proposition~\ref{propositionScaling} and the assumption \(\inf_\Lambda V > 0\).
\end{proof}

\section{Back to the original problem}\label{Section-Back}

\subsection{Linear equation outside small balls}
In order to prove that solutions of the penalized problem \eqref{equationPenalized} solve the original Choquard equation \eqref{equationNLChoquard}, we must show that these solutions are small enough in the penalized region
\(\R^N\setminus\Lambda\).
We begin by showing that if $p\ge 2$ then solutions \(u_\varepsilon\) of penalized problem \eqref{equationPenalized} are also subsolutions of a nonlocal linear problem outside small balls centered on a family of points \((a_\varepsilon)_{\varepsilon > 0}\) at which \(u_\varepsilon\) does not vanish, as constructed in proposition~\ref{propositionConcentration}.

\begin{proposition}[Linear equation outside small balls]
\label{propositionSubsolutions}
Let \(p\in\big[2,\frac{N + \alpha}{(N - 2)_+}\big)\).
For \(\varepsilon > 0\) small enough and \(\delta \in (0,1)\), there exists \(R > 0\), \(a_\varepsilon \in \Lambda\) and \(\nu>0\) such that
\[
\left\{
\begin{aligned}
  - \varepsilon^2 \Delta u_\varepsilon + (1 - \delta) V u_\varepsilon &\le
  \bigl(p \varepsilon^{-\alpha} I_\alpha \ast (H_\varepsilon u_\varepsilon) + \nu \varepsilon^{N - \alpha} I_\alpha \bigr) H_\varepsilon &
  & \text{in \(\R^N \setminus B (a_\varepsilon, R\varepsilon)\)},\\
  u_\varepsilon &\le 1&
  & \text{in \(\Lambda\setminus B (a_\varepsilon, R\varepsilon)\)},
\end{aligned}
\right.
\]
where the constant \(\nu\) only depends on \(V\), \(\Lambda\), \(U\), \(p\) and \(\kappa\).
\end{proposition}
\begin{proof}
Using the assumption \eqref{assumptGBoundp} and since \(p \ge 2\), by proposition~\ref{propositionConcentration}
there exist \(R>0\) and a family of points \((a_\varepsilon)_{\varepsilon > 0}\) in \(\Lambda\) such that
\(p \bigl(\varepsilon^{-\alpha}I_\alpha \ast \mathscr{G}_\varepsilon (u_\varepsilon)\bigr) (u_\varepsilon)_+^{p - 2} \le \delta\) in \(U\setminus B (a_\varepsilon, R\varepsilon)\) and the nonlinear term is bounded as
\[
 p\varepsilon^{-\alpha}\bigl(I_\alpha \ast \mathscr{G}_\varepsilon (u_\varepsilon)\bigr) \mathscr{g}_\varepsilon (u_\varepsilon)
 \le p \bigl(\varepsilon^{-\alpha}I_\alpha \ast \mathscr{G}_\varepsilon (u_\varepsilon)\bigr) (u_\varepsilon)_+^{p - 1}
 \le \delta V u_\varepsilon
 \quad \text{in \(U\setminus B (a_\varepsilon, R\varepsilon)\).}
\]
In view of the assumption \((g_2)\) in \(\R^N\setminus U\), the nonlinear term is bounded there as
\[
 p \bigl(\varepsilon^{-\alpha} I_\alpha \ast \mathscr{G}_\varepsilon (u_\varepsilon)\bigr)\mathscr{g}_\varepsilon (u_\varepsilon)
 \le p \bigl(\varepsilon^{-\alpha} I_\alpha \ast \big(H_\varepsilon u_\varepsilon + \tfrac{1}{p}\chi_{\Lambda} (u_\varepsilon)^p \big)\bigr)H_\varepsilon
  \quad \text{in \(\R^N \setminus U\).}
\]
By the scaled Sobolev inequality proposition~\ref{propositionSobolevScaling} and by the upper bound on the norm of solutions \eqref{lemmaUpperBound-plus}
in \(\R^N\setminus U\) it holds
\begin{equation*}
  \frac{1}{\varepsilon^\alpha} I_\alpha \ast \bigl(\chi_\Lambda (u_\varepsilon)_+^p\bigr)
  \le C \frac{I_\alpha}{\varepsilon^\alpha} \int_{\Lambda} (u_\varepsilon)_+^p \le C' I_\alpha \varepsilon^{N (1 - \frac{p}{2})} \Bigl(\int_{\R^N} \varepsilon^2 \abs{\nabla u_\varepsilon}^2 + V(u_\varepsilon)^2 \Bigr)^\frac{p}{2}
  \le \nu I_\alpha \varepsilon^{N - \alpha},
\end{equation*}
where \(\nu>0\) only depends on the potential \(V\), the sets \(\Lambda\) and \(U\), the exponent \(p\) and the penalization parameter \(\kappa\). We conclude by inserting the previous pointwise bounds in the penalized equation \eqref{equationPenalized}.
\end{proof}

\subsection{Comparison principle}
A second tool is a comparison principle for nonlocal problems
in subdomains of \(\R^N\).

\begin{proposition}[Comparison principle]
\label{propositionComparison}
Let \(\Omega \subseteq \R^N\) be a nonempty open set and \(H\) is nonnegative and that \((H_2)\) holds.
If \(u \in H^1_\mathrm{loc} (\Omega)\cap L^2(\Omega,H_\varepsilon(x)^2\abs{x}^\alpha\dif x)\) satisfies weakly
\[
 -\varepsilon^2 \Delta u + V u \le \frac{p}{\varepsilon^\alpha}(I_\alpha \ast H_\varepsilon u)H_\varepsilon\quad\text{in \(\Omega\)},
\]
that is, if \(\varphi \in H^1 (\Omega)\) is compactly supported in \(\Omega\) and \(\varphi \ge 0\),
\[
  \int_{\Omega} \varepsilon^2 \scalprod{\nabla u}{\nabla \varphi} + V u \varphi \le \frac{p}{\varepsilon^\alpha}\int_{\Omega} \int_{\Omega} I_\alpha (x - y) \, H_\varepsilon (y) u (y) \,H_\varepsilon (x) \varphi (x)\dif y \dif x
\]
and \(u_+ \in H^1_{V, 0} (\Omega)\), then \(u \le 0\) in \(\Omega\).
\end{proposition}

The assumption \(u \in L^2(\Omega,H_\varepsilon (x)^2 \abs{x}^\alpha\dif x)\) ensures that by the weighted Hardy--Littlewood--Sobolev inequality (proposition~\ref{propositionHLSW}) \(I_\alpha \ast H_\varepsilon u \in L^2(\Omega, \abs{x}^{-\alpha}\dif x)\). By \((H_2)\), if \(\varphi \in H^1 (\Omega)\) is compactly supported in \(\Omega\), then \(H_\varepsilon \varphi \in L^2(\Omega, \abs{x}^{\alpha}\dif x)\).

\begin{proof}
By definition of the space \(H^1_{V, 0} (\Omega)\), there exists a sequence \((\varphi_n)_{n \in \N}\) in \(C^\infty_c(\Omega)\) such that $\lim_{n \to \infty} \norm{\varphi_n-u_+}\to 0$.
If we set \(u_n:=\min\{(\varphi_n)_+,u_+\}\), then \(u_n\in H^1_{V, 0} (\Omega)\) and by Lebesgue's dominated convergence theorem,
\[
\lim_{n \to \infty} \int_{\R^N} \varepsilon^2 \abs{\nabla u_n - \nabla u_+}^2 + V \abs{u_n - u_+}^2 = 0.
\]
In addition, \(u_n\) is nonnegative and compactly supported in \(\Omega\) and \(u_- u_n=0\).
Testing the inequation against \(u_n\) we obtain, by the Cauchy-Schwarz inequality, the weighted Hardy--Littlewood--Sobolev inequality (proposition~\ref{propositionHLSW}) and the assumption \((H_2)\),
\begin{equation*}
\begin{split}
0 &\ge \int_{\Omega} \bigl(\varepsilon^2 \scalprod{\nabla u}{\nabla u_n}
      + Vu u_n\bigr) - \frac{p}{\varepsilon^\alpha} \int_{\Omega} (I_\alpha\ast H_\varepsilon u) H_\varepsilon u_n,\\
& \ge \int_{\Omega} \bigl(\varepsilon^2 \scalprod{\nabla u_+}{\nabla u_n} + V u_+ u_n\bigr) - \frac{p}{\varepsilon^\alpha}\int_{\Omega} (I_\alpha\ast H_\varepsilon u_+)H_\varepsilon u_n\\
& \ge \int_{\Omega} \bigl(\varepsilon^2 \scalprod{\nabla u_+}{\nabla u_n} + V u_+ u_n\bigr)\\
&\qquad \qquad
- C_\alpha p \kappa \Bigl(\int_{\Omega} \varepsilon^2 \abs{\nabla u_n}^2 + V \abs{u_n}^2\Bigr)^\frac{1}{2}\Bigl(\int_{\Omega} \varepsilon^2 \abs{\nabla u_+}^2 + V \abs{u_+}^2\Bigr)^\frac{1}{2}.
\end{split}
\end{equation*}
By letting \(n \to \infty\), we conclude that
\[
  0 \ge (1 - C_\alpha \kappa) \int_{\Omega} \abs{\nabla u_+}^2 + V\abs{u_+}^2,
\]
so that, since \(C_\alpha \kappa < 1\),  \(u_+ = 0\) and thus \(u \le 0\).
\end{proof}

\subsection{Construction of penalization from supersolutions}
We now construct supersolutions to a modification of nonlocal equation \eqref{equationNLChoquard}.
These supersolutions will be used to construct explicitly penalizations \(H_\varepsilon\) and to employ proposition
\ref{propositionSubsolutions} in order to show that solutions of penalized problem \eqref{equationPenalized}
solve the original problem \eqref{equationNLChoquard} for sufficiently small \(\varepsilon>0\).

\begin{proposition}[Construction of penalization from supersolutions]
\label{propositionConstruction}
Let \(N \in \N\), \(\alpha \in (0, N)\), \(p\ge 2\) and \(V \in C(\R^N;[0, \infty))\).
If \(p=2\) and \(\alpha \ge N - 2\), assume additionally that \(\inf_{\R^N}V>0\) and that
\[
 \liminf_{\abs{x} \to \infty} V (x) \abs{x}^{N - \alpha} > 0.
\]
If \(2<p \le 1 + \max(\alpha, 1 + \frac{\alpha}{2})/(N - 2)_+\) and \(\alpha > N - 2\), assume additionally that
\[
 \liminf_{\abs{x} \to \infty} V (x) \abs{x}^{2} > 0.
\]
Let \(\nu > 0\), \(\lambda>0\) and \(\delta\in(0,1)\).
For all sufficiently small \(\varepsilon > 0\), there exists \(\Bar{u}_\varepsilon \in H^1_{V} (\R^N)\), \(\Bar{u}_\varepsilon>0\) such that
\begin{equation*}
\left\{
\begin{aligned}
 -\varepsilon^2 \Delta \Bar{u}_\varepsilon + (1-\delta) V \Bar{u}_\varepsilon & \ge
 \bigl( p \varepsilon^{- \alpha}I_\alpha \ast \bigl(\chi_{\R^N \setminus \Lambda} \Bar{u}_\varepsilon^p\bigr) + \nu\varepsilon^{N - \alpha} I_\alpha\bigr) \chi_{\R^N \setminus \Lambda}{\Bar{u}_\varepsilon}^{p - 1}
 & & \text{in \(\R^N\)},\\
 \Bar{u}_\varepsilon & = e^{- \frac{\lambda}{{\varepsilon}}} & &\text{on \(\Bar{\Lambda}\)}.
\end{aligned}
\right.
\end{equation*}
Moreover, \(H_\varepsilon=\chi_{\R^N\setminus\Lambda} \Bar{u}_\varepsilon^{p - 1}\)
defines a penalization which converges uniformly to \(0\) as \(\varepsilon \to 0\) and satisfies \((H_1)\), \((H_2)\).
\end{proposition}

\begin{proof}
We consider separately three different cases.

\begin{case}
\(N\ge 3\) and
\(p > 1 + \max(\alpha, \frac{\alpha + 2}{2})/(N - 2)\).
\end{case}
For \(\mu>0\) we take \(w_\mu \in C^2(\R^N,\R)\) such that \(w_\mu = 1\) on \(\Bar{\Lambda}\), \(w_\mu > 0\) on \(\Bar{U} \setminus \Lambda\) and for every \(x \in \R^N \setminus \Bar{U}\),
\begin{equation}\label{eq:wmu}
 w_\mu (x) = \frac{1}{\abs{x}^\mu}.
\end{equation}
We also take \(\mu \in (0, N-2)\) and, for sufficiently small \(\varepsilon > 0\), we compute
\begin{equation}
\label{eqLaplacewmu}
 - \varepsilon^2 \Delta w_\mu (x) + (1 - \delta) V w_\mu (x) \ge \varepsilon^2 \chi_{\R^N \setminus \Lambda} (x)
 \frac{\mu (N - 2 - \mu)}{\abs{x}^2} w_\mu (x)\quad\text{for \(x \in \R^N\)}.
\end{equation}
Moreover, if \(\mu p > \alpha\), there exists a constant \(C > 0\), such that for every \(x \in \R^N\)
\[
  (I_\alpha \ast w_\mu^p) (x) w_\mu (x)^{p - 2}
  \le \left\{
  \begin{aligned}
     &\dfrac{C}{\abs{x}^{\mu (2 p - 2) - \alpha}} & &\text{if \(\mu p < N\)},\\
     &\dfrac{C \log (\abs{x} + e)}{\abs{x}^{\mu (p - 2) + N - \alpha}} & &\text{if \(\mu p = N\)},\\
     &\dfrac{C}{\abs{x}^{\mu (p - 2) + N - \alpha}} & &\text{if \(\mu p > N\)}.
  \end{aligned}
  \right.
\]
We deduce therefrom that under the assumption \(p > 1 + \max (\alpha, \frac{\alpha + 2}{2})/(N - 2)\), we can choose \(\mu \in (0, N - 2)\) in such a way that for every \(x \in \R^N \setminus \Lambda\),
\begin{align}
\label{eqIalphaw}
  (I_\alpha \ast w_\mu^p)(x) w_\mu (x)^{p - 2}
  &\le \frac{C'}{\abs{x}^2}&
  & \text{ and }&
  I_\alpha (x) w_\mu (x)^{p - 2} &\le \frac{C'}{\abs{x}^2}.
\end{align}
We define \(\Bar{u}_\varepsilon = e^{- \frac{\lambda}{{\varepsilon}}} w_\mu\).
Then for \(\varepsilon>0\) small enough and \(x \in \R^N \setminus \Lambda\), in view of \eqref{eqIalphaw} we have
\[
   p\varepsilon^{-\alpha}\bigl(I_\alpha \ast (\chi_{\R^N \setminus \Lambda} {\Bar{u}_\varepsilon}^p)\Bigr) (x) \Bar{u}_\varepsilon (x)^{p - 2}
   \le  C' \frac{p\varepsilon^{- \alpha} e^{- (2 p - 2) \frac{\lambda}{{\varepsilon}}}}{\abs{x}^2}
   \le \frac{1}{2}\frac{\varepsilon^2}{\abs{x}^2} .
\]
and
\[
  \nu\varepsilon^{N - \alpha} I_\alpha (x) {\Bar{u}_\varepsilon} (x)^{p - 2}
 \le  \chi_{\R^N \setminus \Lambda} (x)C' \frac{\nu \varepsilon^{N - \alpha} e^{- (p - 2) \frac{\lambda}{{\varepsilon}}}}{\abs{x}^2} \le \frac{1}{2}\frac{\varepsilon^2}{\abs{x}^2}.
\]
When \(p = 2\), this is possible since \(2 < N - \alpha\) by our assumption.
Therefore \(\Bar{u}_\varepsilon\) defines the required supersolution.

Next we verify the properties of the penalization \(H_\varepsilon=\chi_{\R^N\setminus\Lambda} \Bar{u}_\varepsilon^{p - 1}\).
Since \(\mu \in (0, N - 2)\) and \eqref{eqIalphaw}, we can choose \(\mu\) in such a way that
\begin{equation}\label{e:case1}
2\mu(p-1)-\alpha>2\max(\alpha, \frac{\alpha + 2}{2})-\alpha>2
\end{equation}
is satisfied.
Then \(x \mapsto H_\varepsilon(x)^2 \abs{x}^\alpha\in L^q(\R^N)\cap L^\infty(\R^N)\), for some \(q<\tfrac{N}{2}\).
and the homogeneous Sobolev space \(H^1_0 (\R^N)(\R^N)\)
is compactly embedded into the space \(L^2(\R^N,(H_\varepsilon (x)^2\abs{x}^\alpha + \chi_\Lambda(x))\dif x)\) \cite{BonheureVanSchaftingen2010}*{theorem 5(ii)}.
Since \(H^1_{V}(\R^N)\) is continuously embedded into \(H^1_0 (\R^N)(\R^N)\), the assumption \((H_1)\) follows.

Further, \eqref{e:case1} implies that for all \(x\in\R^N\),
\begin{equation*}
H_\varepsilon (x)^2\abs{x}^\alpha\le e^{- \frac{2(p-1)\lambda}{{\varepsilon}}}\frac{c}{\abs{x}^2}.
\end{equation*}
For every \(\kappa>0\), by the classical Hardy inequality in \(\R^N\) with \(N\ge 3\), we find \(\varepsilon_0>0\) such that for all \(\varepsilon\in(0,\varepsilon_0]\)
and for all \(\varphi\in C^\infty_c(\R^N)\) it holds
\[
    \frac{1}{\varepsilon^\alpha}\int_{\R^N} \abs{H_\varepsilon (x)\varphi (x)}^2 \abs{x}^\alpha \dif x\le
    \frac{ce^{- \frac{2(p-1)\lambda}{{\varepsilon}}}}{\varepsilon^\alpha}\int_{\R^N} \frac{\abs{\varphi (x)}^2 }{\abs{x}^2}\dif x
        \le \varepsilon^2\kappa \int_{\R^N}\abs{\nabla \varphi}^2,
 \]
which implies \((H_2)\).

\begin{case}
\(p=2\), \(\alpha \ge N - 2\) and \(\inf_{x \in \R^N} V (x) (1 + \abs{x}^{N - \alpha}) > 0\).
\end{case}

Take \(w_\mu\) as defined in \eqref{eq:wmu} with \(\mu > \frac{N}{2}\).
In view of \eqref{eqLaplacewmu}, since \(\inf_{x \in \R^N} V (x) (1 + \abs{x}^{N - \alpha}) > 0\) and \(\alpha \ge N - 2\),
for a fixed \(\delta\in(0,1)\) and all sufficiently small \(\varepsilon > 0\),
\[
 - \varepsilon^2 \Delta w_\mu  + (1 - \delta) V w_\mu \ge \tfrac{1 - \delta}{2} V w_\mu\quad\text{in \(\R^N\)}.
\]
Moreover, there exists \(C > 0\) such that for every \(x \in \R^N\),
\[
  \bigl(I_\alpha \ast w_{\mu}^p\bigr) (x) \le C I_\alpha (x).
\]
We define again \(\Bar{u}_\varepsilon = e^{- \frac{\lambda}{{\varepsilon}}} w_\mu\).
Since \(\inf_{x \in \R^N \setminus \Lambda} V (x) \abs{x}^{N - \alpha} > 0\),  when \(\varepsilon > 0\) is small enough,
we can estimate on \(\R^N \setminus \Lambda\),
\[
\begin{split}
 p\varepsilon^{-\alpha}\bigl(I_\alpha \ast (\chi_{\R^N \setminus \Lambda} {\Bar{u}_\varepsilon}^2)\bigr)
 &\le p\varepsilon^{-\alpha}
 e^{- 2 \frac{\lambda}{{\varepsilon}}} \bigl(I_\alpha \ast (w_{\mu}^p)\bigr) w_{\mu}^{p - 2}  \le C \varepsilon^{- \alpha} e^{- 2 \frac{\lambda}{{\varepsilon}}} I_\alpha \le \tfrac{1 - \delta}{4} V
\end{split}
\]
and
\[
 \nu\varepsilon^{N - \alpha} I_\alpha {\Bar{u}_\varepsilon}^{p - 2}
 \le C \varepsilon^{N - \alpha} I_\alpha
 \le \tfrac{1 - \delta}{4} V.
\]
This implies that \(\Bar{u}_\varepsilon\) is the required supersolution.

Next we verify the properties of the penalization \(H_\varepsilon=\chi_{\R^N\setminus\Lambda} \Bar{u}_\varepsilon^{p - 1}\).
For all \(x\in\R^N\),
\begin{equation*}
H_\varepsilon (x)^2\abs{x}^\alpha\le
\frac{c e^{- \frac{2\lambda}{{\varepsilon}}}}{(1+\abs{x})^{2\mu-\alpha}}.
\end{equation*}
In particular, since \(\mu>\frac{N}{2}\),
\[
  \lim_{\abs{x}\to\infty}\frac{H_\varepsilon (x)^2\abs{x}^\alpha}{V(x)}=0,
\]
and the compactness of the embedding
\(H^1_{V}(\R^N)\subset L^2(\R^N,(H_\varepsilon (x)^2\abs{x}^\alpha + \chi_\Lambda(x))\dif x)\)
follows from \cite{BonheureVanSchaftingen2010}*{theorem 4(ii)} when \(N\ge 3\), or
\cite{BonheureVanSchaftingen2010}*{p.287} for \(N=1,2\).
This settles \((H_1)\).

Let \(\kappa>0\). Since \(\alpha \ge N - 2\) and \(\inf_{\R^N} V (x) (1 + \abs{x}^{N - \alpha}) > 0\),
we find \(\varepsilon_0>0\)
such that for all \(\varepsilon\in(0,\varepsilon_0]\) and for all \(\varphi\in C^\infty_c(\R^N)\) it holds
\[
    \frac{1}{\varepsilon^\alpha}\int_{\R^N} \abs{H_\varepsilon (x)\varphi (x)}^2 \abs{x}^\alpha \dif x\le
    \frac{ce^{- \frac{2\lambda}{{\varepsilon}}}}{\varepsilon^\alpha}\int_{\R^N} \frac{\abs{\varphi (x)}^2 }{(1+\abs{x})^{N-\alpha}}\dif x
        \le \kappa \int_{\R^N}V \abs{\varphi}^2,
 \]
which implies \((H_2)\).

\begin{case}
\(p > 2\) and \(\liminf_{\abs{x} \to \infty} V (x)\abs{x}^2 > 0\).
\end{case}

Choose \(R > 0\) such that \(\inf_{x \in \R^N \setminus B_R} V (x) \abs{x}^2 > 0\) and \(\bar U\subset B_R\). For \(\mu > 0\) choose \(w_\mu \in C^2 (\R^N)\) such that \(w_\mu > 0\), \(w_\mu = 1\) on \(\Lambda\), \(w_\mu (x) = 2R^2 - \abs{x}^2\) on \(B_{R} \setminus U\)
and \(w_{\mu} (x) = \abs{x}^{-\mu}\) on \(\R^N \setminus B_{2 R}\). For \(\varepsilon > 0\) small enough, we have
\[
  -\varepsilon^2 \Delta w_\mu + (1 - \delta) V w_\mu
  \ge \varepsilon^2 \chi_{B_{R} \setminus U} \tfrac{N}{R^2} w_\mu + \tfrac{1- \delta}{2} V w_\mu\quad\text{in \(\R^N\)}.
\]
If \(\mu p > N\) then there exists \(C > 0\) such that on \(\R^N\)
\[
  I_\alpha \ast w_\mu^p \le C I_\alpha.
\]
We define
\(
 \Bar{u}_\varepsilon
 = e^{- \frac{\lambda}{{\varepsilon}}} w_\mu.
\) and we observe that if we choose \(\mu (p - 2) + N - \alpha > 2\) then in \(\R^N \setminus \Lambda\)
we can estimate for sufficiently small \(\varepsilon>0\),
\[
\begin{split}
 p\varepsilon^{-\alpha}\bigl(I_\alpha \ast (\chi_{\R^N \setminus \Lambda} {\Bar{u}_\varepsilon}^p)\Bigr)\Bar{u}_\varepsilon^{p - 2}
 &\le p\varepsilon^{-\alpha}
 e^{- (2p - 2) \frac{\lambda}{{\varepsilon}}} \bigl(I_\alpha \ast (w_{\mu}^p)\bigr) w_{\mu}^{p - 2} \\
 & \le C \varepsilon^{- \alpha} e^{- (2 p - 2) \frac{\lambda}{{\varepsilon}}} I_\alpha w_{\mu}^{p - 2}\\
 & \le \frac{1}{2} \bigl(\varepsilon^2 \chi_{B_{R} \setminus U} \tfrac{N}{R^2}  + (1- \delta) V \bigr).
\end{split}
\]
and
\[
 \nu\varepsilon^{N - \alpha} I_\alpha {\Bar{u}_\varepsilon}^{p - 2}
 \le C \varepsilon^{N - \alpha} e^{- (p - 2) \frac{\lambda}{{\varepsilon}}} I_\alpha w_{\mu}^{p - 2}
 \le \frac{1}{2} \bigl(\varepsilon^2 \chi_{B_{R} \setminus U} \tfrac{N}{R^2}  + (1- \delta) V \bigr).
\]
The function \(\Bar{u}_\varepsilon\) defines the required supersolution.

Define the penalization \(H_\varepsilon=\chi_{\R^N\setminus\Lambda} \Bar{u}_\varepsilon^{p - 1}\).
Note  that \(\mu\) in the definition of \(w_\mu\) was chosen so that \(\mu p>N\) and
\(\mu (p - 2) + N - \alpha > 2\). These two assumptions imply that
\[
2\mu(p-1)-\alpha>2.
\]
The latter ensures that
\begin{equation}\label{e:case3}
\lim_{\abs{x}\to\infty}\frac{H_\varepsilon (x)^2\abs{x}^\alpha}{V(x)}=0.
\end{equation}
Let \(R>0\) be such that \(\inf_{x \in \R^N \setminus B_R} V (x) \abs{x}^2 > 0\) and \(\bar U\subset B_R\).
Set \(\bar{V}(x)=V(x)+\chi_{B_R}(x)\), so that \(\bar{V}(x)>0\) for all \(x\in\R^N\).
Then the compactness of the embedding
\(H^1_{\Bar{V}}(\R^N)\subset L^2(\R^N,(H_\varepsilon (x)^2\abs{x}^\alpha + \chi_\Lambda(x))\dif x)\)
follows from \eqref{e:case3}, see \cite{BonheureVanSchaftingen2010}*{theorem 4(ii)} if \(N\ge 3\),
or \cite{BonheureVanSchaftingen2010}*{p.287} for \(N=1,2\).
Using proposition~\ref{propositionSobolevScaling} with \(q=2\) and \(\Lambda=B_R\),
we conclude that \(H^1_V(\R^N)\) is continuously embedded into \(H^1_{\Bar{V}}(\R^N)\)
and the assumption \((H_1)\) follows.

Let \(\kappa>0\). Using \eqref{e:case3} we find \(\varepsilon_0>0\)
such that for all \(\varepsilon\in(0,\varepsilon_0]\) and for all \(\varphi\in C^\infty_c(\R^N)\) it holds
\[
    \frac{1}{\varepsilon^\alpha}\int_{\R^N} \abs{H_\varepsilon (x)\varphi (x)}^2 \abs{x}^\alpha \dif x\le
    \frac{ce^{- \frac{2(p-1)\lambda}{{\varepsilon}}}}{\varepsilon^\alpha}\int_{\R^N} \frac{\abs{\varphi (x)}^2 }{(1+\abs{x})^{2}}\dif x
        \le \frac{\kappa}{2} \int_{\R^N}\Bar{V}  \abs{\varphi}^2.
 \]
Using again proposition~\ref{propositionSobolevScaling} with \(q=2\) and \(\Lambda=B_R\), we obtain
\[
\frac{\kappa}{2} \int_{\R^N}\Bar{V}\varphi^2 \dif x\le
\kappa\Big( \int_{\R^N}\varepsilon^2\abs{\nabla \varphi}^2+V\varphi^2 \Big),
\]
which concludes the proof of \((H_2)\).
\end{proof}

\subsection{Proof of the existence theorems}

We are going to show that the family of  supersolutions \(\Bar{u}_\varepsilon\) constructed in the previous section could be extended to a family of supersolutions in \(\R^N\setminus B (a_\varepsilon, R\varepsilon)\) which are separated away from zero
on \(a_\varepsilon\) and controlled outside \(\Lambda\) by the penalizations \(H_\varepsilon\). These properties are sufficient
in order to use them as the barriers for the family of solutions \(u_\varepsilon\) of the penalized equation \eqref{equationPenalized}

\begin{proposition}[Construction of barrier functions]
\label{propositionBarriers}
In addition to all the assumptions of proposition~\ref{propositionConstruction},
let \(R > 0\) and \((a_\varepsilon)_{\varepsilon>0}\) be a family of points in \(\Lambda\) such that
\(\liminf_{\varepsilon \to 0}d(a_\varepsilon, \partial \Lambda)>0\).
Then for all sufficiently small \(\varepsilon > 0\), there exists \(\Bar{U}_\varepsilon \in H^1_{V} (\R^N)\cap C^{1,1}(\R^N)\), \(\Bar{U}_\varepsilon>0\) such that
\[
\left\{
\begin{aligned}
 -\varepsilon^2 \Delta \Bar{U}_\varepsilon + (1-\delta) V \Bar{U}_\varepsilon & \ge
 \bigl(p \varepsilon^{-\alpha} I_\alpha \ast (H_\varepsilon \Bar{U}_\varepsilon) + \nu \varepsilon^{N - \alpha} I_\alpha \bigr) H_\varepsilon
 & & \text{in \(\R^N\setminus B (a_\varepsilon, R\varepsilon)\)},\\
 \Bar{U}_\varepsilon & \ge 1 & &\text{on \(B (a_\varepsilon, R\varepsilon)\)}.
\end{aligned}
\right.
\]
Moreover, \(\Bar{U}_\varepsilon^{p-1}<H_\varepsilon\) in \(\R^N\setminus\Lambda\).
\end{proposition}

\begin{proof}
We choose \(m > 0\) so that
\[
 m^2 < (1-\delta) \inf_\Lambda V,
\]
\(r>0\) such that
\[
r <\frac{1}{2}\liminf_{\varepsilon \to 0}d(a_\varepsilon, \partial \Lambda).
\]
Let \(\Bar{u}_\varepsilon\) be a family of supersolutions of given by proposition~\ref{propositionConstruction} for some \(\lambda < mr\).
We define
\[
h_\varepsilon(y)=
\begin{cases}
\cosh \frac{m(r-\abs{y})}{\varepsilon} & \text{if \(x \in B(a_\varepsilon, r)\)}, \smallskip\\
1 & \text{if \(x \in \R^N\setminus B(a_\varepsilon, r)\)}.
\end{cases}
\]
Then \(h_\varepsilon\in C^{1, 1}(\R^N)\) and
\begin{equation}
\label{heps}
 -\varepsilon^2 \Delta h_\varepsilon +m^2\chi_{B_r}h_\varepsilon \ge 0\quad\text{in \(\R^N\)}.
\end{equation}
Let \(\Bar{u}_\varepsilon\) be the supersolution, constructed in Proposition~\ref{propositionConstruction}.
Define
\begin{equation*}
\Bar{U}_\varepsilon(x)=2\frac{\Bar{u}_\varepsilon(x)h_\varepsilon(x-a_\varepsilon)}{e^{-\frac{\lambda}{{\varepsilon}}}\cosh \frac{m r}{\varepsilon}}.
\end{equation*}
For small \(\varepsilon>0\) we have \(\Bar{U}_\varepsilon(x)\ge 1\) in \(B(a_\varepsilon,R\varepsilon)\) and \(B(a_\varepsilon, 2r) \subset \Lambda\),
so using the construction of \(\Bar{u}_\varepsilon\) and \eqref{heps} we obtain
\begin{align*}
  - \varepsilon^2 \Delta \Bar{U}_\varepsilon + (1 - \delta) V \Bar{U}_\varepsilon &\ge
  \bigl(p \varepsilon^{-\alpha} I_\alpha \ast (H_\varepsilon \Bar{U}_\varepsilon) + \nu \varepsilon^{N - \alpha} I_\alpha \bigr) H_\varepsilon &
  & \text{in \(\R^N \setminus B (a_\varepsilon, R\varepsilon)\)}.
\end{align*}
Clearly, \(\Bar{U}_\varepsilon^{p-1}\le H_\varepsilon=\Bar{u}_\varepsilon^{p-1}\) in \(\R^N\setminus\Lambda\) for all sufficiently small \(\varepsilon>0\).
\end{proof}

\begin{proof}[Proof of theorems~\ref{theoremLinear} and \ref{theoremSuperlinear}]
Propositions \ref{propositionSubsolutions} and \ref{propositionBarriers} imply
via comparison principle of proposition~\ref{propositionComparison},
that for all sufficiently small \(\varepsilon>0\) solutions \(u_\varepsilon\) of the penalized equation \eqref{equationPenalized}
satisfy
\[
  u_\varepsilon\le \Bar{U}_\varepsilon\quad\text{in } \R^N \setminus B (a_\varepsilon, R\varepsilon).
\]
Since \(\Bar{U}_\varepsilon^{p-1}\le H_\varepsilon\) in \(\R^N\setminus\Lambda\), we conclude that
\(u_\varepsilon\) is a solution of the original problem \eqref{equationNLChoquard}.
\end{proof}

\section{Nonexistence of concentrating solutions}

\subsection{Critical potential well at a point}

In this section we show that for \(p = 2\), if \(\alpha > N - 2\) and if the potential \(V\) vanishes at some point \(a_* \in \R^N\) strongly enough, then the problem \eqref{equationNLChoquard} cannot have a family of solutions that concentrates somewhere in \(\R^N \setminus \{a_*\}\).
Indeed, we would expect a family of solutions that concentrates on some compact set \(K \subset \R^N\) to satisfy
\[
 \liminf_{\varepsilon \to 0} \frac{1}{\varepsilon^N} \int_{K} \abs{u_\varepsilon}^2 > 0,
\]
while we obtain a contradicting asymptotic estimate.

\begin{proposition}\label{propositionNonConcentration}
Let \(N \in \N\), \(\alpha \in (0, N)\), \(p = 2\), \(V \in C(\R^N;[0, \infty))\).
Assume that
\(\alpha > N - 2\) and, as \(\rho \to 0\),
\[
 \frac{1}{\rho^N} \int_{B_\rho (a)} V (x) = o (\rho^{\frac{4}{\alpha + 2 - N} - 2}).
\]
If \((u_\varepsilon)_{\varepsilon \in (0, \varepsilon_0)}\)
is a family in \(\subset H^1_\mathrm{loc}(\R^N) \cap L^2((1+\abs{x})^{-(N-\alpha)}\dif x)\) of positive solutions to \eqref{equationNLChoquard}, then
for every compact set \(K \subset \R^N\setminus\{a_*\}\), as \(\varepsilon\to 0\),
\[
 \int_{K} \abs{u_\varepsilon}^2 = o (\varepsilon^N).
\]
\end{proposition}

The assumptions of this proposition~are satisfied in particular if \(\alpha \ge N - \frac{2}{3}\) and \(V \in C^1 (\R^N)\) vanishes at \(a_*\), or if \(\alpha > N - 1\) and \(V \in C^2 (\R^N)\) vanishes at \(a_*\).

\begin{proof}[Proof of proposition~\ref{propositionNonConcentration}]
We first apply a ground-state transformation to \eqref{equationNLChoquard}, that is, given \(\varphi \in C^\infty_c (\R^N)\), we test the equation
\eqref{equationNLChoquard} against the compactly supported function \(\varphi^2/u_\varepsilon \in H^1 (\R^N)\) to obtain the inequality \cite{Agmon1983}*{theorems 3.1 and 3.3} (see also \cite{MVSGNLSNE}*{proposition~3.1})
\[
 \int_{\R^N} \varepsilon^2 \abs{\nabla \varphi}^2 + V \abs{\varphi}^2
 \ge \varepsilon^{-\alpha} \int_{\R^N} \bigl(I_\alpha \ast \abs{u_\varepsilon}^2 \bigr) \varphi^2.
\]
We choose now the function \(\psi \in C^\infty_c (\R^N) \setminus \{0\}\) such that \(\supp \varphi\subset B_1\). We apply the previous inequality to the function
\(\varphi_\rho : \R^N \to \R\) defined for \(\rho > 0\) and \(x \in \R^N\) by
\[
 \varphi_\rho(x) = \psi \Bigl(\frac{x - a_*}{\rho}\Bigr),
\]
and we deduce if \(B_{2 \rho} (a_*) \cap K = \emptyset\) that for every \(\varepsilon > 0\),
\[
  \frac{1}{\varepsilon^N} \int_{K} \abs{u_\varepsilon}^2
 \le C \frac{\varepsilon^{2 + \alpha - N}}{\rho^{2}} + \frac{1}{\varepsilon^{N - \alpha}} o (\rho^{\frac{4}{\alpha + 2 - N} - 2}),
\]
as \(\rho \to 0\) uniformly in \(\varepsilon > 0\).
Given \(\eta > 0\), we take \(\rho = \varepsilon^\frac{\alpha + 2 - N}{2}/\eta^\frac{1}{2}\) in the previous inequality and we obtain
the asymptotic estimate
\[
  \frac{1}{\varepsilon^N} \int_{K} \abs{u_\varepsilon}^2
 \le C \eta  + o (\eta^{\frac{2}{\alpha + 2 - N} - 1}),
\]
as \((\varepsilon, \eta) \to (0, 0)\).
Since \(\alpha < N\), \(\frac{2}{\alpha + 2 - N} - 1 > 0\) and the conclusion follows.
\end{proof}

In the limiting case \(\alpha = N - 2\), the same techniques limits the mass available in \(K\).

\begin{proposition}\label{propositionNonConcentrationLimit}
Let \(N \in \N\), \(\alpha = N - 2\), \(p = 2\). There exists a constant \(C > 0\) such that for every \(V \in C(\R^N;[0, \infty))\) such that \(V = 0\) on \(B_{a_*} (\rho)\) for some \(a_* \in \R^N\) and \(\rho > 0\), for every positive solution \(u_\varepsilon\in H^1_\mathrm{loc}(\R^N) \cap L^2((1+\abs{x})^{-2}\dif x)\)
of \eqref{equationNLChoquard} and for every compact set \(K \subset \R^N\), one has
\[
 \frac{1}{\varepsilon^N} \int_{K} \abs{u_\varepsilon}^2 \le \frac{C \dist (B_\rho (a), K)^2}{\rho^2 }.
\]
\end{proposition}
\begin{proof}
As in the proof of proposition~\ref{propositionNonConcentration}, we obtain by a ground-state transformation applied to \eqref{equationNLChoquard} for every \(\varphi \in C^\infty_c (\R^N)\) the inequality
\[
 \int_{\R^N} \varepsilon^2 \abs{\nabla \varphi}^2 + V \abs{\varphi}^2
 \ge \varepsilon^{-(N - 2)} \int_{\R^N} \bigl(I_\alpha \ast \abs{u_\varepsilon}^2 \bigr) \varphi^2,
\]
We apply the previous inequality to the function
\(\varphi_\rho : \R^N \to \R\) defined for \(\rho > 0\) and \(x \in \R^N\) by
\[
 \varphi_\rho(x) = \psi \Bigl(\frac{x - a_*}{\rho}\Bigr),
\]
where \(\psi \in C^\infty_c (\R^N) \setminus \{0\}\) is a fixed function, and we conclude that
\[
  \frac{1}{\varepsilon^N} \int_{K} \abs{u_\varepsilon}^2
 \le \frac{C \dist (B_\rho (a), K)^2}{\rho^2 }.\qedhere
\]
\end{proof}

\subsection{Critical potential well at infinity}

Finally, we prove that for \(p = 2\) and \(\alpha = N - 2\),
if the potential \(V\) decays at infinity faster then the inverse square,
then \eqref{equationNLChoquard} cannot have concentrating solutions.
In this case, the nonlocal term forces the rescaled mass to vanish on every compact subset \(K\subset\R^N\).

\begin{proposition}
\label{propositionStrongWellCritical}
Let \(N \ge 3\), \(p = 2\), \(V \in C(\R^N;[0, \infty))\).
Assume that \(\alpha = N - 2\) and
\[
 \lim_{R \to \infty} \frac{1}{R^{N - 2}} \int_{B_{2 R} \setminus B_R} V = 0.
\]
If \(u_\varepsilon\in H^1_\mathrm{loc}(\R^N) \cap L^2((1+\abs{x})^{-2}\dif x)\)  is a solution to \eqref{equationNLChoquard},
then
\[
 \frac{1}{\varepsilon^N} \int_{\R^N} \abs{u_\varepsilon}^2 \le \Gamma(\tfrac{N - 2}{2})\pi^{N/2}2^{N - 2} \Bigl(\frac{N - 2}{2} \Bigr)^2.
\]
\end{proposition}
\begin{proof}
As in the proof of proposition~\ref{propositionNonConcentration}, by a ground-state transformation,
we have for every \(\varphi \in C^\infty_c (\R^N)\) the inequality
\[
 \int_{\R^N} \varepsilon^2 \abs{\nabla \varphi}^2 + V \abs{\varphi}^2
 \ge \varepsilon^{-\alpha} \int_{\R^N} \bigl(I_\alpha \ast \abs{u_\varepsilon}^2 \bigr) \varphi^2.
\]
For every \(\varphi \in C^\infty_c (\R^N\setminus\{0\}) \setminus \{0\}\), we apply the previous inequality to
\(\varphi_R\) defined for \(R > 0\) and \(x \in \R^N\) by
\[
 \varphi_R(x) = \varphi \Bigl(\frac{x}{R}\Bigr).
\]
We observe that by change of variable and our assumption on \(V\),
\[
 \lim_{R \to \infty} \frac{1}{R^{N - 2}}\int_{\R^N} \varepsilon^2 \abs{\nabla \varphi_R}^2 + V \abs{\varphi_R}^2= \int_{\R^N} \varepsilon^2 \abs{\nabla \varphi}^2.
\]
On the other hand, by a change of variable,
\[
 \frac{1}{R^{N - 2}} \int_{\R^N} \bigl(I_{N - 2} \ast \abs{u_\varepsilon}^2 \bigr) \varphi_R^2
 = \int_{\R^N} \frac{1}{R^2} \bigl(I_{N - 2} \ast \abs{u_\varepsilon}^2 \bigr) (R z) \abs{\varphi (z)}^2\dif z.
\]
We take into account the homogeneity of degree \(-2\) of the Riesz potential \(I_{N - 2}\) to rewrite the convolution and we apply Fatou's lemma to conclude that for every \(z \in \R^N \setminus \{0\}\),
\[
 \liminf_{R \to \infty} \frac{1}{R^2} \bigl(I_{N - 2} \ast \abs{u_\varepsilon}^2 \bigr) (R z)
 = \liminf_{R \to \infty}  \int_{\R^N} I_{N - 2} (z - w/R) \abs{u_\varepsilon (w)}^2 \dif z
 \ge I_{N - 2} (z)\int_{\R^N} \abs{u_\varepsilon}^2.
\]
We obtain thus, by Fatou's lemma again by letting \(R \to \infty\),
\[
  \int_{\R^N} \abs{u_\varepsilon}^2 \int_{\R^N} I_{N - 2} \abs{\varphi}^2
  \le \varepsilon^N \int_{\R^N} \abs{\nabla \varphi}^2.
\]
The conclusion follows by using the optimal constant in the Hardy inequality (see e.g. \cite{Willem2013}*{theorem 6.4.10 and exercise 6.8}).
\end{proof}

It is not clear whether this condition effectively obstructs the construction of concentrating solutions. Indeed, if \(v_\lambda\) is a solution of \eqref{equationLimit} such that \(\mathcal{I}_\lambda (v_\lambda) = \mathcal{E} (\lambda)\), then by proposition~\ref{propositionStrongWellCritical} and the Poho\v zaev identity for \eqref{equationLimit} \cite{MVSGNLSNE}*{Proposition 3.2},
\[
  \int_{\R^N} \abs{v_\lambda}^2 = \frac{2}{\lambda} \mathcal{E} (\lambda) = 2 \mathcal{E} (1).
\]
Further advancement on the problem would require to estimate sharply enough \(\mathcal{E} (1)\).

\section{Concluding remarks}
Here we list several questions related to the existence of concentrating solutions
for the Choquard equation \eqref{equationNLChoquard} which are left open in this work.

\subsection{Concentrating solutions in the locally sublinear case $p<2$}
The functional associated to the Choquard equation \eqref{equationNLChoquard} is a  \emph{superquadratic} perturbation of a local quadratic form for every \(p>1\).
The limiting equation \eqref{eqLimit} is variationally well--posed for \(p\in\big(\frac{N + \alpha}{N},\frac{N + \alpha}{(N - 2)_+}\big)\), see for example \cite{MVSNENLSNE}.
Although we prove the existence of solutions of the penalized problem for $p\in\big(1,\frac{N + \alpha}{(N - 2)_+}\big)$ (see proposition \ref{propositionPenalizedExistence}),
our penalization  methods covers only the locally superlinear range \(p\ge 2\).
In fact the only essential step in our considerations which requires $p\ge 2$ is the construction of the linear equation outside small balls (proposition~\ref{propositionSubsolutions}).

The existence of families of concentrating positive solutions for \eqref{equationNLChoquard}
in the locally sublinear range \(p\in\big(\frac{N+\alpha}{N},2\big)\) remains open and seems to be a delicate problem.

\subsection{Concentration in the presence of weak critical potential wells}
When \(p=2\), \(\alpha > N - 2\) and \(V\) vanishes at some points in \(\R^N\setminus\Lambda\)
at a rate weaker then the strong critical potential well assumption \eqref{vanishingrate}
of theorem \ref{theoremNonConcentrationSpecific}, the existence of families of positive solutions for \eqref{equationNLChoquard} which concentrate in \(\Lambda\) remains open.
In this case it might be possible to construct solutions concentrating in \(\Lambda\) by gluing the solutions of a suitably penalized nonlocal problem with the first Neumann eigenfunction of
\[
  -\varepsilon^2 \Delta u (x) + c \abs{x - a_*}^{\frac{4}{\alpha + 2 - N}-2} u (x) = \varepsilon^{-\alpha} u (x)\quad\text{in \(B_{\varepsilon\rho}(a_*)\).}
\]
Similar approach should be possible if \(\alpha = N - 2\) and \(\liminf_{x \to a} V (x)/\abs{x}^{\gamma} > 0\) for some \(\gamma > 0\).

A more difficult problem is to tackle the case when \(\alpha = N - 2\) and \(V\) is allowed to take any nonnegative value outside \(\Lambda\), even to vanish on some open set.

\subsection{Concentration around critical potential wells}
The existence of families of positive solutions for the Choquard equation \eqref{equationNLChoquard} which concentrate around critical potential wells of \(V\) remains open. Relevant results obtained in the case of
the local nonlinear Schr\"odinger equations \citelist{\cite{ByeonWang2003}\cite{ByeonWang2002}} suggest that the form of the limiting equation
should strongly depend on the rate at which the potential \(V\) vanishes at the point.

\subsection{Multipeak solutions and concentration around critical point or critical manifolds}
In the case  \(N = 3\), \(\alpha = 2\), \(p = 2\), families of solutions for the Choquard equation
which concentrate to nondegenerate critical points of \(V\) have been constructed
in \citelist{\cite{WeiWinter2009}\cite{Secchi2010}} via a Lyapunov--Schmidt type reduction techniques.
In addition, \cite{WeiWinter2009} establishes the existence of \emph{multi--peak} solutions.
The penalization method developed in this work should be applicable for the study of families
of concentrating and multipeak solutions for \eqref{equationNLChoquard} located around local maxima
or critical points of \(V\); as well as around model critical manifolds,
such as curves and lower--dimensional spheres. In the local context of equation \eqref{equationLocal}
related results were obtained within the relevant penalization framework in
\citelist{\cite{BonheureDiCosmoVanSchaftingen2012}\cite{DiCosmoVanSchaftingen}}
(see also \citelist{\cite{delPinoFelmer1997}
\cite{delPinoFelmer1998}
\cite{AmbrosettiMalchiodi2006}
\cite{AmbrosettiMalchiodi2007}
\cite{CingolaniSecchi}}).

\subsection{Concentrating solutions via the Lyapunov--Schmidt reduction}
Families of concentrating solutions for the Choquard equation have been constructed
via a \emph{Lyapunov--Schmidt type reduction} only in the case  \(N = 2\), \(\alpha = 2\), \(p = 2\)
and under some restrictions on the decay of the potential \(V\), see \citelist{\cite{WeiWinter2009}\cite{Secchi2010}}. It is an interesting question wh	ether the Lyapunov--Schmidt reduction techniques
could be extended to the framework of the present work, that is, to a
natural range of exponents \(p\) and to the classes of potentials \(V\)
with optimal decay assumptions at infinity. The difficulty in applying small perturbation techniques
is largely due to the fact that in the case of fast decaying or compactly supported potentials
the original Choquard equation \eqref{equationNLChoquard} and the limit equation \eqref{eqLimit} are well--posed in different functional spaces, which limits applicability of standard small perturbation methods.
In the context of local equation \eqref{equationLocal} an Lyapunov--Schmidt type reduction approach
which allows to handle this issue was recently developed in \cite{Kwon-2012}.


\begin{bibdiv}
\begin{biblist}

\bib{Adams1975}{book}{
   author={Adams, Robert A.},
   title={Sobolev spaces},
   series={Pure and Applied Mathematics},
   volume={65},
   publisher={Academic Press},
   address={New York-London},
   date={1975},
   pages={xviii+268},
}

\bib{Agmon1983}{article}{
    AUTHOR = {Agmon, Shmuel},
     TITLE = {On positivity and decay of solutions of second order elliptic
              equations on {R}iemannian manifolds},
    conference={
      title={Methods of functional analysis and theory of elliptic
              equations},
      address={Naples},
      date={1982},
   },
   book={
      publisher={Liguori},
      address={Naples},
      date={1983},
   },
     PAGES = {19--52},
}

\bib{AmbrosettiBadialeCingolani}{article}{
   author={Ambrosetti, A.},
   author={Badiale, M.},
   author={Cingolani, S.},
   title={Semiclassical states of nonlinear Schr\"odinger equations},
   journal={Arch. Rational Mech. Anal.},
   volume={140},
   date={1997},
   number={3},
   pages={285--300},
   issn={0003-9527},
}

\bib{AmbrosettiFelliMalchiodi2005}{article}{
   author={Ambrosetti, Antonio},
   author={Felli, Veronica},
   author={Malchiodi, Andrea},
   title={Ground states of nonlinear Schr\"odinger equations with potentials
   vanishing at infinity},
   journal={J. Eur. Math. Soc. (JEMS)},
   volume={7},
   date={2005},
   number={1},
   pages={117--144},
   issn={1435-9855},
}

\bib{AmbrosettiMalchiodi2006}{book}{
   author={Ambrosetti, Antonio},
   author={Malchiodi, Andrea},
   title={Perturbation methods and semilinear elliptic problems on ${\mathbf
   R}^n$},
   series={Progress in Mathematics},
   volume={240},
   publisher={Birkh\"auser Verlag},
   place={Basel},
   date={2006},
   pages={xii+183},
   isbn={978-3-7643-7321-4},
   isbn={3-7643-7321-0},
}

\bib{AmbrosettiMalchiodi2007}{article}{
   author={Ambrosetti, Antonio},
   author={Malchiodi, Andrea},
   title={Concentration phenomena for nonlinear Schr\"odinger equations:
   recent results and new perspectives},
   book={
      title={Perspectives in nonlinear partial differential equations},
      series={Contemp. Math.},editor={Berestycki, Henri},
      editor={Bertsch, Michiel},
      editor = {Browder, Felix E.},
      editor = {Nirenberg, Louis},
      editor = {Peletier, Lambertus A.},
      editor = {V\'eron, Laurent},
      volume={446},
      publisher={Amer. Math. Soc.},
      place={Providence, RI},
      date={2007},
   },
   pages={19--30},
}

\bib{AmbrosettiProdi1993}{book}{
   author={Ambrosetti, Antonio},
   author={Prodi, Giovanni},
   title={A primer of nonlinear analysis},
   series={Cambridge Studies in Advanced Mathematics},
   volume={34},
   publisher={Cambridge University Press},
   place={Cambridge},
   date={1993},
   pages={viii+171},
   isbn={0-521-37390-5},
}

\bib{AmbrosettiRabinowitz1973}{article}{
   author={Ambrosetti, Antonio},
   author={Rabinowitz, Paul H.},
   title={Dual variational methods in critical point theory and
   applications},
   journal={J. Functional Analysis},
   volume={14},
   date={1973},
   pages={349--381},
}
		



\bib{Appell-Zabreiko}{book}{
   author={Appell, J{\"u}rgen},
   author={Zabrejko, Petr P.},
   title={Nonlinear superposition operators},
   series={Cambridge Tracts in Mathematics},
   volume={95},
   publisher={Cambridge University Press},
   place={Cambridge},
   date={1990},
   pages={viii+311},
   isbn={0-521-36102-8},
}

\bib{BonheureDiCosmoVanSchaftingen2012}{article}{
   author={Bonheure, Denis},
   author={Di Cosmo, Jonathan},
   author={Van Schaftingen, Jean},
   title={Nonlinear Schr\"odinger equation with unbounded or vanishing
   potentials: solutions concentrating on lower dimensional spheres},
   journal={J.~Differential Equations},
   volume={252},
   date={2012},
   number={2},
   pages={941--968},
   issn={0022-0396},
}

\bib{BonheureVanSchaftingen2006}{article}{
   author={Bonheure, Denis},
   author={Van Schaftingen, Jean},
   title={Nonlinear Schr\"odinger equations with potentials vanishing at
   infinity},
   journal={C. R. Math. Acad. Sci. Paris},
   volume={342},
   date={2006},
   number={12},
   pages={903--908},
   issn={1631-073X},
}
		
\bib{BonheureVanSchaftingen2008}{article}{
   author={Bonheure, Denis},
   author={Van Schaftingen, Jean},
   title={Bound state solutions for a class of nonlinear Schr\"odinger
   equations},
   journal={Rev. Mat. Iberoam.},
   volume={24},
   date={2008},
   number={1},
   pages={297--351},
   issn={0213-2230},
}

\bib{BonheureVanSchaftingen2010}{article}{
   author={Bonheure, Denis},
   author={Van Schaftingen, Jean},
   title={Groundstates for the nonlinear Schr\"odinger equation with
   potential vanishing at infinity},
   journal={Ann. Mat. Pura Appl. (4)},
   volume={189},
   date={2010},
   number={2},
   pages={273--301},
   issn={0373-3114},
}
		
\bib{ByeonWang2002}{article}{
   author={Byeon, Jaeyoung},
   author={Wang, Zhi-Qiang},
   title={Standing waves with a critical frequency for nonlinear
   Schr\"odinger equations},
   journal={Arch. Ration. Mech. Anal.},
   volume={165},
   date={2002},
   number={4},
   pages={295--316},
   issn={0003-9527},
}

\bib{ByeonWang2003}{article}{
   author={Byeon, Jaeyoung},
   author={Wang, Zhi-Qiang},
   title={Standing waves with a critical frequency for nonlinear
   Schr\"odinger equations. II},
   journal={Calc. Var. Partial Differential Equations},
   volume={18},
   date={2003},
   number={2},
   pages={207--219},
   issn={0944-2669},
}

\bib{CingolaniClappSecchi2012}{article}{
   author={Cingolani, Silvia},
   author={Clapp, M{\'o}nica},
   author={Secchi, Simone},
   title={Multiple solutions to a magnetic nonlinear Choquard equation},
   journal={Z. Angew. Math. Phys.},
   volume={63},
   date={2012},
   number={2},
   pages={233--248},
   issn={0044-2275},
}

\bib{CingolaniJeanjeanSecchi2009}{article}{
   author={Cingolani, Silvia},
   author={Jeanjean, Louis},
   author={Secchi, Simone},
   title={Multi-peak solutions for magnetic NLS equations without
   non-degeneracy conditions},
   journal={ESAIM Control Optim. Calc. Var.},
   volume={15},
   date={2009},
   number={3},
   pages={653--675},
   issn={1292-8119},
}

\bib{CingolaniSecchi}{article}{
   author={Cingolani, Silvia},
   author={Secchi, Simone},
   title={Multiple \(\mathbb{S}^{1}\)-orbits for the Schr\"odinger-Newton system},
   journal={Differential and Integral Equations},
   volume={26},
   number={9/10},
   pages={867--884},
   date={2013},
}

\bib{CingolaniSecchiSquassina2010}{article}{
   author={Cingolani, Silvia},
   author={Secchi, Simone},
   author={Squassina, Marco},
   title={Semi-classical limit for Schr\"odinger equations with magnetic
   field and Hartree-type nonlinearities},
   journal={Proc. Roy. Soc. Edinburgh Sect. A},
   volume={140},
   date={2010},
   number={5},
   pages={973--1009},
   issn={0308-2105},
}

\bib{ClappSalazar}{article}{
   author={Clapp, M{\'o}nica},
   author={Salazar, Dora},
   title={Positive and sign changing solutions to a nonlinear Choquard
   equation},
   journal={J. Math. Anal. Appl.},
   volume={407},
   date={2013},
   number={1},
   pages={1--15},
   issn={0022-247X},
}

\bib{delPinoFelmer1997}{article}{
   author={del Pino, Manuel},
   author={Felmer, Patricio L.},
   title={Semi-classical states for nonlinear Schr\"odinger equations},
   journal={J.~Funct. Anal.},
   volume={149},
   date={1997},
   number={1},
   pages={245--265},
   issn={0022-1236},
}

\bib{delPinoFelmer1998}{article}{
   author={del Pino, Manuel},
   author={Felmer, Patricio L.},
   title={Multi-peak bound states for nonlinear Schr\"odinger equations},
   journal={Ann. Inst. H. Poincar\'e Anal. Non Lin\'eaire},
   volume={15},
   date={1998},
   number={2},
   pages={127--149},
   issn={0294-1449},
}


\bib{DiCosmoVanSchaftingen}{article}{
  author = {Di Cosmo, Jonathan},
  author = {Van Schaftingen, Jean},
  title = {Stationary solutions of the nonlinear Schr\"odinger equation with fast-decay potentials concentrating around local maxima},
  journal = {Calc. Var. Partial Differential Equations},
  volume={47},
  date={2013},
  number={1--2},
  pages={243-271},
}
		


\bib{FloerWeinstein}{article}{
   author={Floer, Andreas},
   author={Weinstein, Alan},
   title={Nonspreading wave packets for the cubic Schr\"odinger equation
   with a bounded potential},
   journal={J. Funct. Anal.},
   volume={69},
   date={1986},
   number={3},
   pages={397--408},
   issn={0022-1236},
}
		
\bib{GenevVenkov2012}{article}{
   author={Genev, Hristo},
   author={Venkov, George},
   title={Soliton and blow-up solutions to the time-dependent
   Schr\"odinger-Hartree equation},
   journal={Discrete Contin. Dyn. Syst. Ser. S},
   volume={5},
   date={2012},
   number={5},
   pages={903--923},
   issn={1937-1632},
}



\bib{He1977}{article}{
   author={Herbst, Ira W.},
   title={Spectral theory of the operator
   \((p^{2}+m^{2})^{1/2}-Ze^{2}/r\)},
   journal={Comm. Math. Phys.},
   volume={53},
   date={1977},
   number={3},
   pages={285--294},
   issn={0010-3616},
}


\bib{KRWJones1995gravitational}{article}{
  title={Gravitational Self-Energy as the Litmus of Reality},
  author={Jones, K. R. W.},
  journal={Modern Physics Letters A},
  volume={10},
  number={8},
  pages={657--668},
  year={1995},
  publisher={Singapore: World Scientific, c1986-}
}

\bib{KRWJones1995newtonian}{article}{
  title={Newtonian Quantum Gravity},
  author={Jones, K. R. W.},
  journal={Australian Journal of Physics},
  volume={48},
  number={6},
  pages={1055-1081},
  year={1995},
}


\bib{Kwon-2012}{article}{
   author={Kwon, Ohsang},
   title={Existence of standing waves of nonlinear Schr\"odinger equations
   with potentials vanishing at infinity},
   journal={J. Math. Anal. Appl.},
   volume={387},
   date={2012},
   number={2},
   pages={920--930},
   issn={0022-247X},
}

\bib{Lieb1977}{article}{
   author={Lieb, Elliott H.},
   title={Existence and uniqueness of the minimizing solution of Choquard's
   nonlinear equation},
   journal={Studies in Appl. Math.},
   volume={57},
   date={1976/77},
   number={2},
   pages={93--105},
}

\bib{LiebLoss2001}{book}{
   author={Lieb, Elliott H.},
   author={Loss, Michael},
   title={Analysis},
   series={Graduate Studies in Mathematics},
   volume={14},
   edition={2},
   publisher={American Mathematical Society},
   place={Providence, RI},
   date={2001},
   pages={xxii+346},
   isbn={0-8218-2783-9},
}

\bib{Lions1980}{article}{
   author={Lions, P.-L.},
   title={The Choquard equation and related questions},
   journal={Nonlinear Anal.},
   volume={4},
   date={1980},
   number={6},
   pages={1063--1072},
   issn={0362-546X},
}

\bib{Lions1984-1}{article}{
   author={Lions, P.-L.},
   title={The concentration-compactness principle in the calculus of
   variations. The locally compact case.},
   part = {I},
   journal={Ann. Inst. H. Poincar\'e Anal. Non Lin\'eaire},
   volume={1},
   date={1984},
   number={2},
   pages={109--145},
   issn={0294-1449},
}

\bib{Ma-Zhao-2010}{article}{
   author={Ma Li},
   author={Zhao Lin},
   title={Classification of positive solitary solutions of the nonlinear
   Choquard equation},
   journal={Arch. Ration. Mech. Anal.},
   volume={195},
   date={2010},
   number={2},
   pages={455--467},
   issn={0003-9527},
}

\bib{Menzala1980}{article}{
   author={Menzala, Gustavo Perla},
   title={On regular solutions of a nonlinear equation of Choquard's type},
   journal={Proc. Roy. Soc. Edinburgh Sect. A},
   volume={86},
   date={1980},
   number={3-4},
   pages={291--301},
   issn={0308-2105},
}

\bib{Menzala1983}{article}{
   author={Menzala, Gustavo Perla},
   title={On the nonexistence of solutions for an elliptic problem in
   unbounded domains},
   journal={Funkcial. Ekvac.},
   volume={26},
   date={1983},
   number={3},
   pages={231--235},
   issn={0532-8721},
}

\bib{Moroz-Penrose-Tod-1998}{article}{
   author={Moroz, Irene M.},
   author={Penrose, Roger},
   author={Tod, Paul},
   title={Spherically-symmetric solutions of the Schr\"odinger-Newton
   equations},
   journal={Classical Quantum Gravity},
   volume={15},
   date={1998},
   number={9},
   pages={2733--2742},
   issn={0264-9381},
}


\bib{MorozVanSchaftingen2009}{article}{
   author={Moroz, Vitaly},
   author={Van Schaftingen, Jean},
   title={Existence and concentration for nonlinear Schr\"odinger equations
   with fast decaying potentials},
   journal={C. R. Math. Acad. Sci. Paris},
   volume={347},
   date={2009},
   number={15-16},
   pages={921--926},
   issn={1631-073X},
}
		
\bib{MorozVanSchaftingen2010}{article}{
   author={Moroz, Vitaly},
   author={Van Schaftingen, Jean},
   title={Semiclassical stationary states for nonlinear Schr\"odinger
   equations with fast decaying potentials},
   journal={Calc. Var. Partial Differential Equations},
   volume={37},
   date={2010},
   number={1-2},
   pages={1--27},
   issn={0944-2669},
}


\bib{MVSHardy}{article}{
  author = {Moroz, Vitaly},
  author = {Van Schaftingen, Jean},
  title = {Nonlocal Hardy type inequalities with optimal constants and remainder terms},
  journal = {Ann. Univ. Buchar. Math. Ser.},
  volume={3 (LXI)},
  number={2},
  year={2012},
  pages={187--200},
}

\bib{MVSNENLSNE}{article}{
   author={Moroz, Vitaly},
   author={Van Schaftingen, Jean},
   title={Nonexistence and optimal decay of supersolutions to Choquard
   equations in exterior domains},
   journal={J. Differential Equations},
   volume={254},
   date={2013},
   number={8},
   pages={3089--3145},
   issn={0022-0396},
}


\bib{MVSGNLSNE}{article}{
   author={Moroz, Vitaly},
   author={Van Schaftingen, Jean},
   title={Groundstates of nonlinear Choquard equations: Existence,
   qualitative properties and decay asymptotics},
   journal={J. Funct. Anal.},
   volume={265},
   date={2013},
   number={2},
   pages={153--184},
   issn={0022-1236},
}


\bib{MVSGGCE}{article}{
  author = {Moroz, Vitaly},
  author = {Van Schaftingen, Jean},
  title = { Existence of groundstates for a class of nonlinear Choquard equations},
  note = {to appear in Trans. Amer. Math. Soc.},
  eprint={arXiv:1212.2027}
}




\bib{Pekar1954}{book}{
   author={Pekar, S.},
   title={Untersuchung {\"u}ber die Elektronentheorie der Kristalle},
   publisher={Akademie Verlag},
   place={Berlin},
   date={1954},
   pages={184},
}

\bib{Penrose1996}{article}{
   author={Penrose, Roger},
   title={On gravity's role in quantum state reduction},
   journal={Gen. Relativity Gravitation},
   volume={28},
   date={1996},
   number={5},
   pages={581--600},
   issn={0001-7701},
}


\bib{PinchoverTintarev2006}{article}{
   author={Pinchover, Yehuda},
   author={Tintarev, Kyril},
   title={A ground state alternative for singular Schr\"odinger operators},
   journal={J. Funct. Anal.},
   volume={230},
   date={2006},
   number={1},
   pages={65--77},
   issn={0022-1236},
}

\bib{Rabinowitz1986}{book}{
   author={Rabinowitz, Paul H.},
   title={Minimax methods in critical point theory with applications to
   differential equations},
   series={CBMS Regional Conference Series in Mathematics},
   volume={65},
   publisher={Published for the Conference Board of the Mathematical
   Sciences, Washington, DC},
   date={1986},
   pages={viii+100},
   isbn={0-8218-0715-3},
}


\bib{Riesz}{article}{
   author={Riesz, Marcel},
   title={L'int\'egrale de Riemann-Liouville et le probl\`eme de Cauchy},
   journal={Acta Math.},
   volume={81},
   date={1949},
   pages={1--223},
   issn={0001-5962},
}

\bib{Schwartz1969}{book}{
   author={Schwartz, J. T.},
   title={Nonlinear functional analysis},
   publisher={Gordon and Breach},
   place={New York},
   date={1969},
   pages={vii+236},
}

\bib{Secchi2010}{article}{
   author={Secchi, Simone},
   title={A note on Schr\"odinger-Newton systems with decaying electric
   potential},
   journal={Nonlinear Anal.},
   volume={72},
   date={2010},
   number={9-10},
   pages={3842--3856},
   issn={0362-546X},
}


\bib{Stein-Weiss}{article}{
   author={Stein, E. M.},
   author={Weiss, Guido},
   title={Fractional integrals on \(n\)-dimensional Euclidean space},
   journal={J. Math. Mech.},
   volume={7},
   date={1958},
   pages={503--514},
}

\bib{Struwe2008}{book}{
   author={Struwe, Michael},
   title={Variational methods},
   series={Ergebnisse der Mathematik und ihrer Grenzgebiete. 3. Folge},
   volume={34},
   edition={4},
   subtitle={Applications to nonlinear partial differential equations and
   Hamiltonian systems},
   publisher={Springer},
   place={Berlin},
   date={2008},
   pages={xx+302},
   isbn={978-3-540-74012-4},
}

\bib{Tod-Moroz-1999}{article}{
   author={Tod, Paul},
   author={Moroz, Irene M.},
   title={An analytical approach to the Schr\"odinger-Newton equations},
   journal={Nonlinearity},
   volume={12},
   date={1999},
   number={2},
   pages={201--216},
   issn={0951-7715},
}

\bib{WeiWinter2009}{article}{
   author={Wei, Juncheng},
   author={Winter, Matthias},
   title={Strongly interacting bumps for the Schr\"odinger-Newton equations},
   journal={J.~Math. Phys.},
   volume={50},
   date={2009},
   number={1},
   pages={012905, 22},
   issn={0022-2488},
}

\bib{WillemMinimax}{book}{
   author={Willem, Michel},
   title={Minimax theorems},
   series={Progress in Nonlinear Differential Equations and their
   Applications, 24},
   publisher={Birkh\"auser},
   place={Boston, Mass.},
   date={1996},
   pages={x+162},
}

\bib{Willem2013}{book}{
  author = {Willem, Michel},
  title = {Functional analysis},
  subtitle = {Fundamentals and Applications},
  series={Cornerstones},
  publisher = {Birkh\"auser},
  place = {Basel},
  volume = {XIV},
  pages = {213},
  date={2013},
}

\bib{YinZhang2009}{article}{
   author={Yin, Huicheng},
   author={Zhang, Pingzheng},
   title={Bound states of nonlinear Schr\"odinger equations with potentials
   tending to zero at infinity},
   journal={J. Differential Equations},
   volume={247},
   date={2009},
   number={2},
   pages={618--647},
   issn={0022-0396},
}
\end{biblist}

\end{bibdiv}
\end{document}